\documentclass[12pt,twoside]{amsart}
\usepackage[margin=3cm]{geometry}
\usepackage[colorlinks=false]{hyperref}
\usepackage[english]{babel}
\usepackage{graphicx,titling}
\usepackage{float}
\usepackage{amsmath,amsfonts,amssymb,amsthm}
\usepackage{lipsum}
\usepackage[T1]{fontenc}
\usepackage{fourier}
\usepackage{color}
\usepackage[latin1]{inputenc}
\usepackage{esint}
\usepackage{caption}
\usepackage{ccicons}

\makeatletter
\def\blfootnote{\xdef\@thefnmark{}\@footnotetext}
\makeatother

\newcommand\ccnote{
    \blfootnote{\copyright\,\, Otis Chodosh, Yevgeny Liokumovich, and Luca Spolaor}
    \blfootnote{\ccLogo\, \ccAttribution\,\, Licensed under a \href{https://creativecommons.org/licenses/by/4.0/}{Creative Commons Attribution License (CC-BY)}.}
}

\usepackage[export]{adjustbox}
\numberwithin{equation}{section}
\usepackage{setspace}

\renewcommand{\leq}{\leqslant}

\renewcommand{\geq}{\geqslant}
\renewcommand{\mathbb}{\varmathbb}
\usepackage{fancyhdr}
\newtheorem{theorem}{Theorem}[section]
\newtheorem{lemma}[theorem]{Lemma}

\newtheorem{proposition}[theorem]{Proposition}
\newtheorem{definition}[theorem]{Definition}
\newtheorem{remark}[theorem]{Remark}
\usepackage{amsfonts}
\usepackage{amssymb}
\usepackage{graphicx}
\usepackage{amsmath}
\usepackage{latexsym}
\usepackage{amscd}
\usepackage{xypic}
\usepackage[all,cmtip]{xy}
\usepackage{mathrsfs}
\usepackage{enumitem} 
\usepackage{braket}

\usepackage{tikz} \usetikzlibrary{arrows,snakes,shapes,decorations.markings,patterns}
\usepackage{imakeidx}
\makeindex[title=Index of Notation]

\usepackage{hyperref}

\allowdisplaybreaks

\newcommand{\NN}{\mathbb{N}}

\newcommand{\RR}{\mathbb{R}}

\newcommand{\ZZ}{\mathbb{Z}}

\newcommand{\C}{{\bf C}}
\newcommand{\Snm}{\mathcal{S}_{\textnormal{nm}}}
\newcommand{\Om}{\Omega}

\newcommand{\cA}{\mathcal A}

\newcommand{\cC}{\mathcal C}

\newcommand{\cF}{\mathcal F}
\newcommand{\cG}{\mathcal G}
\renewcommand{\cH}{\mathcal H}
\newcommand{\cI}{\mathcal I}

\newcommand{\cM}{\mathcal M}

\newcommand{\cO}{\mathcal O}

\renewcommand{\cR}{\mathcal R}
\newcommand{\cS}{\mathcal S}

\newcommand{\cV}{\mathcal V}

\newcommand{\bC}{\mathbf{C}}
\newcommand{\mm}{{\bf m}}

\newcommand{\bM}{\mathbf{M}}

\DeclareMathOperator{\Per}{Per}
\DeclareMathOperator{\Vol}{Vol}

\DeclareMathOperator{\supp}{supp}

\DeclareMathOperator{\Ric}{Ric}

\newcommand{\eps}{\varepsilon}

\DeclareMathOperator{\Index}{Index}
\DeclareMathOperator{\sing}{Sing} 
\DeclareMathOperator{\reg}{Reg}

\newcommand\res{\mathop{\hbox{\vrule height 7pt width .3pt depth 0pt
\vrule height .3pt width 5pt depth 0pt}}\nolimits}

\DeclareMathOperator{\Met}{Met}
\DeclareMathOperator{\hnm}{\mathfrak{h}_\textnormal{nm}}

\address{OC: Department of Mathematics, Stanford University, Building 380, Stanford, California 94305, USA.}
\email{ochodosh@stanford.edu}
\medskip
\address{YL: Department of Mathematics, University of Toronto, 40 St
George Street, Toronto, ON M5S 2E4, Canada.}
\email{ylio@math.toronto.edu}
\medskip
\address{LS: Department of Mathematics, UC San Diego, AP\&M, La Jolla, California, 92093, USA.}
\email{lspolaor@ucsd.edu}

\begin{document}

\thispagestyle{empty}

\begin{minipage}{0.28\textwidth}
\begin{figure}[H]
\includegraphics[width=2.5cm,height=2.5cm,left]{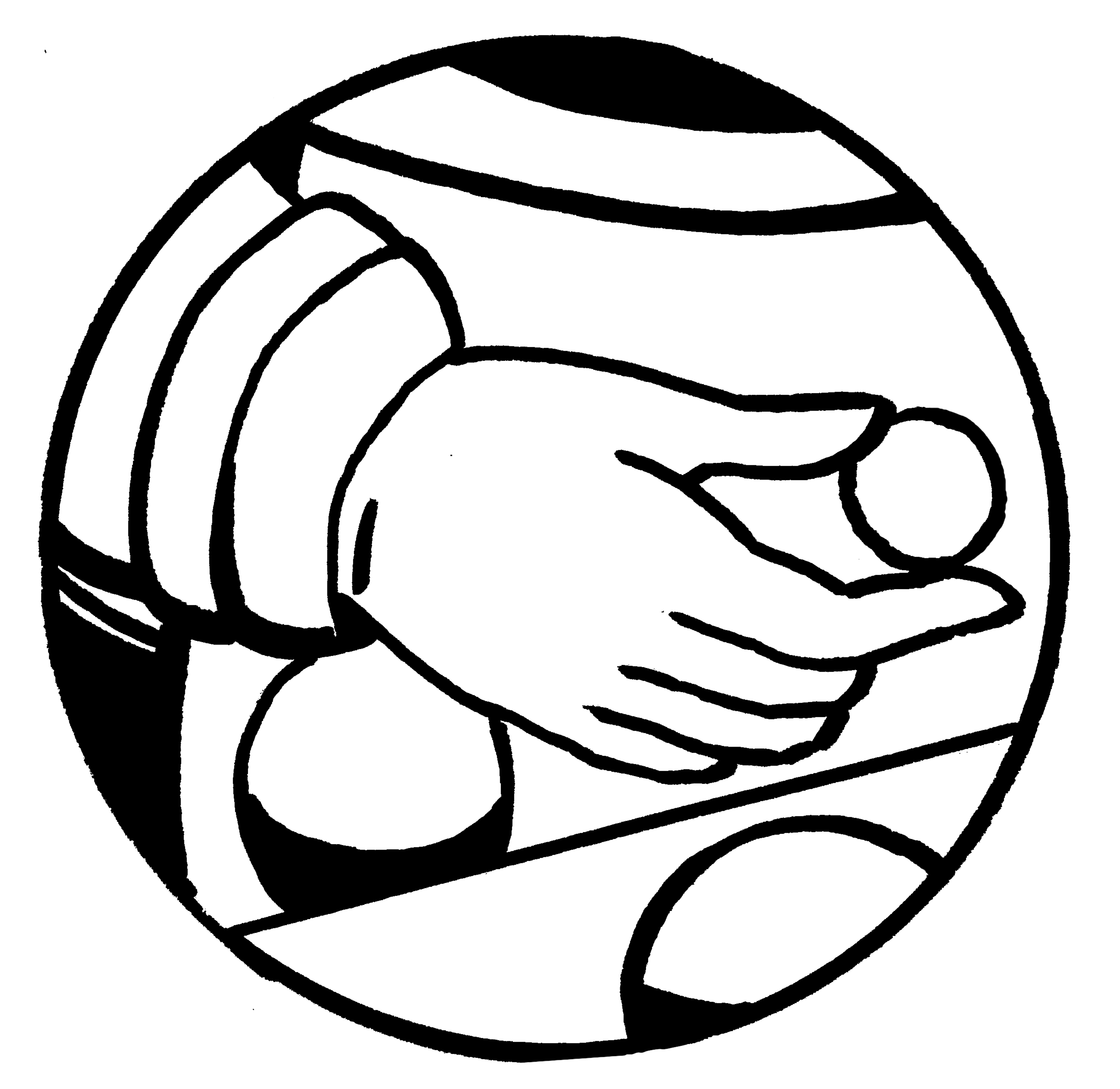}
\end{figure}
\end{minipage}
\begin{minipage}{0.7\textwidth} 
\begin{flushright}
Ars Inveniendi Analytica (2022), Paper No. 2, 27 pp.
\\
DOI 10.15781/j4aj-kd66
\\
ISSN: 2769-8505
\end{flushright}
\end{minipage}

\ccnote

\vspace{1cm}


\begin{center}
\begin{huge}
\textit{Singular behavior and generic regularity of min-max minimal hypersurfaces}

\end{huge}
\end{center}

\vspace{1cm}


\begin{minipage}[t]{.28\textwidth}
\begin{center}
{\large{\bf{Otis Chodosh}}} \\
\vskip0.15cm
\footnotesize{Stanford University}
\end{center}
\end{minipage}
\hfill
\noindent
\begin{minipage}[t]{.28\textwidth}
\begin{center}
{\large{\bf{Yevgeny Liokumovich}}} \\
\vskip0.15cm
\footnotesize{University of Toronto}
\end{center}
\end{minipage}
\hfill
\noindent
\begin{minipage}[t]{.28\textwidth}
\begin{center}
{\large{\bf{Luca Spolaor}}} \\
\vskip0.15cm
\footnotesize{UC San Diego} 
\end{center}
\end{minipage}

\vspace{1cm}


\begin{center}
\noindent \em{Communicated by Francesco Maggi}
\end{center}
\vspace{1cm}


\noindent \textbf{Abstract.} \textit{We show that for a generic $8$-dimensional Riemannian manifold with positive Ricci curvature, there exists a smooth minimal hypersurface. Without the curvature condition, we show that for a dense set of $8$-dimensional Riemannian metrics there exists a minimal hypersurface with at most one singular point. This extends previous work on generic regularity that only dealt with area-minimizing hypersurfaces.}

\textit{These results are a consequence of a more general estimate for a one-parameter min-max minimal hypersurface $\Sigma \subset (M,g)$ (valid in any dimension):
\[
\cH^{0} (\mathcal{S}_{\textnormal{nm}}(\Sigma)) +\Index(\Sigma) \leq 1
\]
where $\mathcal{S}_{\textnormal{nm}}(\Sigma)$ denotes the set of singular points of $\Sigma$ with a unique tangent cone non-area minimizing on either side.}
\vskip0.3cm

\vspace{0.5cm}



\section*{Introduction}

It is well known that $7$-dimensional area minimizing hypersurfaces can have isolated singularities. Using work of Hardt--Simon \cite{HS}, Smale proved in \cite{Smale} that in an $8$-dimensional manifold $M$ with $H_7(M; \mathbb{Z}) \neq 0$, there exists a smooth embedded area minimizing hypersurface for a generic choice of metric. In other words, he showed that isolated singularities of an area-minimizing $7$-dimensional hypersurface can generically be perturbed away. 

One may thus seek to find a smooth embedded minimal hypersurface in all $8$-manifolds $M$ equipped with a generic metric $g$, even when $H_7(M;\mathbb{Z})=0$. Here, we find such a hypersurface in the case of positive Ricci curvature, and give a partial answer in general. We let $\Met^{2,\alpha}(M)$ denote the space of Riemannian metrics of regularity $C^{2,\alpha}$ on $M$ and $\Met^{2,\alpha}_{\Ric>0}(M)\subset\Met^{2,\alpha}(M)$ denote the open subset of Ricci positive metrics. 

\begin{theorem}[Generic regularity with positive Ricci in dimension $8$]\label{thm:generic_bound_ricci}
Let $M^8$ be a compact smooth $8$-manifold. There is an open and dense set $\cG \subset \Met_{\Ric>0}^{2,\alpha}(M)$ so that for $g\in \cG$, there exists a smooth embedded minimal hypersurface $\Sigma \subset M$. 
\end{theorem}

Without the curvature condition, we have the following partial result. 

  \begin{theorem}[Generic almost regularity in dimension $8$]\label{thm:generic_bound}
  	Let $M^8$ be a compact smooth $8$-manifold. There exists a dense set $\cG \subset \Met^{2,\alpha}(M)$ so that for $g \in \cG$,  there exists a smooth embedded minimal hypersurface $\Sigma \subset M$ with at most one singular point. 
  \end{theorem}
 
 We actually prove more general results valid in all dimensions, see Theorem \ref{thm:generic_stratum} below. 
  
  As mentioned above, the principal motivation for such results is to study generic regularity of non-minimizing, high-dimensional minimal submanifolds. This contrasts with previous works on generic regularity:
\begin{itemize}
    \item Hardt--Simon \cite{HS} (resp.\ Smale \cite{Smale}), cf.\ \cite{Liu}, show that regular singularities of (one-sided) minimizing hypersurfaces can be perturbed away by perturbing the boundary (resp.\ metric). 
    \item White \cite{White:85,white2019generic} shows that minimizing integral 2-cycles are smoothly embedded surfaces for a generic metric.  
    \item Moore \cite{Moore,Moore:book} shows that parametrized minimal (2-dimensional) surfaces are free of branch points for a generic ambient metric. 
\end{itemize}

 \medskip

In fact, our work proves that generically there exists a minimal hypersurface of optimal regularity avoiding \emph{certain} singularities in ambient dimensions beyond the singular dimension. Indeed, Theorem \ref{thm:generic_bound} is a consequence of a more general result stated below.

\begin{theorem}[Generic removability of isolated singularities]\label{thm:generic_stratum}
Consider a compact smooth $(n+1)$-manifold, for $n\geq 7$. There is a dense set $\cG\subset \Met^{2,\alpha}(M)$ with the following properties:
\begin{itemize}
\item If $g\in\cG$ then there exists a minimal hypersurface $\Sigma$, smooth away from a closed singular set of Hausdorff dimension at most $n-7$, so that for $\cS_0 \subset \sing(\Sigma)$ the set of singular points with regular tangent cones, we have $\cH^0(S_0) \leq 1$. 
\item If $g \in \cG \cap \Met^{2,\alpha}_{\Ric>0}(M)$ then the same statement holds, except we can conclude that $\cH^0(\cS_0) = 0$. 
\end{itemize}
%
\end{theorem}

 \medskip

In order to remove the topological condition $H_7(M; \mathbb{Z}) \neq 0$ of Smale, we will use the Almgren--Pitts min-max construction \cite{Pitts}, which guarantees the existence of a minimal hypersurface $\Sigma^n$ in a closed Riemannian manifold $(M^{n+1},g)$. As in the area-minimizing case, when the dimension $n$ satisfies $2\leq n \leq 6$, the Almgren--Pitts minimal hypersurface is smooth, but for larger values of $n$ there may be an at most $(n-7)$-dimensional singular set (this follows from work of Schoen--Simon \cite{SS}). However tangent cones to min-max hypersurfaces are \emph{a priori} only stable, while only area-minimizing cones have complements that are foliated by smooth minimal hypersurfaces (cf.\ \cite{BDG, Lawlor}) and it seems that such a foliation is needed (at least on one side) to perturb the singularity away by adjusting the metric  \cite{HS}.

The key technical result of this paper is that (for one-parameter min-max) at all points---except possibly one---of the singular set with a regular tangent cone, the tangent cone is area minimizing on at least one side. Put another way, we show that tangent cones that are not area minimizing on either side ``contribute to the Morse index'' from the point of view of min-max (and these are precisely the cones that we are unable to perturb away using Hardt--Simon \cite{HS}). 

\subsection{Detailed description of results} 
Let $(M^{n+1},g)$ be a closed Riemannian manifold. By a \emph{sweepout} of $M$ we will mean a family of (possibly singular) hypersurfaces $\{ \Phi(x)=\partial \Omega(x)\}_{x \in [0,1]}$, where each hypersurface $\Phi(x)$ is the boundary of an open set
$\Omega(x)$ with $\Omega(0) = \emptyset$ and $\Omega(1) = M$, and we denote the family of such sweepouts by $\cS$ (see Section \ref{sec:min-max} for the precise definition).
The width, $W(M)$, is then defined by 
\[
W(M) = \inf_{\Phi\in \cS} \left\{ \sup_x \bM(\Phi(x))\right\}\,.
\]

  Given a stationary integral varifold $V$, with $\supp V$ regular outside of a set of $n-7$ Hausdorff dimension, we define 
 \[
\mathfrak{h}_\textnormal{nm}(V):= \left\{ p\in \supp(V)\,:\,  \begin{gathered}
\textrm{for all $r>0$ small, $\supp V\cap B_r(p)$ is not one-sided} \\
\text{homotopy area minimizing on either side (in $B_r(p)$).}
\end{gathered}
\right\}\,
 \]
  
In other words, $p\in\mathfrak{h}_\textrm{nm}(V)$ implies that in any small ball there are one-sided homotopies on both sides of $\supp V$ that strictly decrease area without ever increasing area. 
Let $\cR$ denote the set of integral varifolds, 
whose support is a complete embedded minimal hypersurface regular away from a closed singular set of Hausdorff dimension $n-7$.
Finally, we let  $\Index(V)$ denote the Morse index of the regular part of the support of $V$, that is
$$
\Index (V)=\Index (\supp(\reg(V)))\,.
$$
Then the main technical estimate of this paper is the following result.

\begin{theorem}[Index plus non-area minimizing singularities bound]\label{thm:nm+index_bound}
For $n\geq 7$, let $(M^{n+1}, g)$ be a closed Riemannian manifold of class $C^2$. There exists a stationary integral varifold
	$V\in \cR$ such that $|V|(M)=W$, which satisfies 
\begin{equation}\label{e:bound1}
 \cH^{0} (\mathfrak{h}_\textnormal{nm}(V)) +\Index(V) \leq 1\,.
\end{equation}
If equality holds in \eqref{e:bound1}, then for any point $p\in \supp V\setminus \mathfrak{h}_\textnormal{nm}(V)$ there is $\eps>0$ so that $\supp V$ is area-minimizing to one side in $B_\eps(p)$. Finally, we can write $V=\sum_{i}\kappa_i\,|\Sigma_i|$, where $\Sigma_i$ are finitely many disjoint embeddded minimal hypersursufaces smooth away from finitely many points with  $\kappa_i\leq 2$ for every $i$; if $\Sigma_i$ is one-sided then $\kappa_i =2$ and if $\kappa_j=2$ for some $j$ then each $\Sigma_i$ is stable. 
\end{theorem}

The above bound is valid in all dimensions and can be seen as a generalization of the work of Calabi--Cao concerning min-max on surfaces \cite{CaCa}. Indeed if we define  $\Snm(V)$ by\footnote{Here $\omega$ is a modulus of continuity, and we could take it to be logarithmic, as suggested by the work of \cite{Simon}.  Notice in fact that at all isolated singularities $\cS_0$, minimal surfaces have unique tangent cone and are locally $C^{1,\log}$ deformation of the cone itself.}
\[
\Snm(V):= \left\{ p\in \supp(V)\,:\, \begin{gathered}
V \text{ is locally a $C^{1,\omega}$ graph over its \emph{unique} tangent cone $\cC$}\\
\text{at $p$ and both sides of $\cC$ are not one-sided minimizing} 
\end{gathered}  
\right\}\,
\]
then we will see that $\Snm(V) \subset \mathfrak{h}_\textrm{nm}(V)$ in Lemma \ref{lem:snm-hnm}. In particular, \eqref{e:bound1} implies that
\[
\cH^0(\Snm(V)) + \Index(V) \leq 1.
\]
Thus, if we are guaranteed to have $\Index(V) =1$ (e.g., in positive curvature) we see that $\Snm(V) = \emptyset$. This is precisely the higher dimensional analogue of the result of Calabi--Cao (cf. Figure \ref{fig:starfish} and the remark below). 

See also the more recent work of Mantoulidis \cite{mantoulidis} which makes a more explicit connection with Morse index, using the Allen--Cahn approach (as developed by Guaraco and Gaspar \cite{Guaraco,GasparGuaraco}) rather than Almgren--Pitts; it would be interesting to elucidate the relationship between Mantoulidis's Allen--Cahn techniques and our proof of Theorem \ref{thm:nm+index_bound}.

\begin{figure}[h]
	\centering
	\includegraphics[scale=0.4]{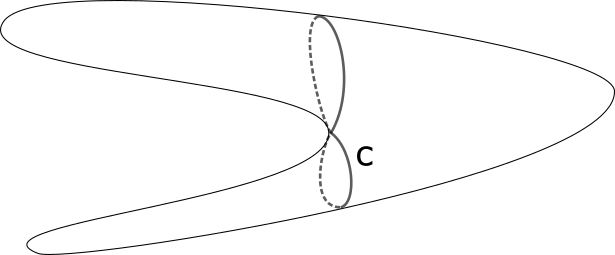}
	\caption{The figure eight geodesic $c$ is an 
	example of a min-max closed geodesic
	that is stable and has one singularity 
	with non-area minimizing tangent cone.}
	\label{fig:starfish}
\end{figure}

\begin{remark}
By the index bound in Theorem \ref{thm:nm+index_bound}, any tangent cone to $V$ has stable regular part. Moreover, we note that the Simons cones \cite{Simons} in $\mathbb{R}^8$ (formed from products of two spheres) are all stable and area minimizing on (at least) one side (cf.\ \cite{Lawlor}). We particularly emphasize that the Simons cone 
\[
\mathbf{C}^{1,5}: = \{(x,y) \in \RR^2\times \RR^6 : 5|x|^2 = |y|^2\}
\]
is one-sided minimizing (and stable), but is not minimizing on the other side. It seems to be an open question whether or not there exists an $n$-dimensional stable cone that does not minimize area on either side, for $n\geq 7$. 

Even assuming the existence of a stable minimal cone which is not area minimizing on either sides, it is hard to decide if the above bound is optimal. In dimension $n=1$, such an example is provided by the classical starfish example (cf.\ Figure \ref{fig:starfish}), whose tangent cone at the singular point (the union of two lines through the origin) is indeed stable non-area minimizing on either sides (and the starfish fails to be one-sided homotopy minimizing on either side).

We conjecture that if there is a regular stable minimal cone that is not area-minimizing on either side, then it can arise as the tangent cone to a min-max minimal hypersurface (possibly in a manifold geometrically similar to the starfish); note that were this to occur, Theorem \ref{thm:nm+index_bound} would imply that the resulting hypersurface would necessarily be stable.
\end{remark}

Theorem \ref{thm:nm+index_bound} generalizes the index upper bound of Marques and Neves \cite{MN16} for Riemannian manifolds $M^{n+1}$, $3\leq n+1\leq 7$ (see also \cite{Zhou-reg-ind}).
In recent years there has been tremendous progress in the understanding of the geometry of minimal hypersurfaces constructed using min-max methods in
these dimensions
(see \cite{DL}, \cite{MN19}, \cite{CM}, \cite{Zh19}, \cite{So} and references therein).

For manifolds of dimension $n+1\geq 8$ much less is known. 
When Ricci curvature is positive
Zhou obtained index and multiplicity bounds for one-parameter min-max minimal hypersurface \cite{Z17} (see also the work of  Ram\'irez-Luna \cite{RL} and Bellettini \cite{bellettini}). Upper Morse index bounds are known to hold in arbitrary manifolds of any dimensions for hypersurfaces constructed by Allen--Cahn, as proven by Hiesmayr and Gaspar \cite{Hiesmayr,Gaspar} (see also the recent work of Dey showing that the Almgren--Pitts and Allen--Cahn approaches are equivalent \cite{dey}). Li proved \cite{li2019} existence 
of infinitely many distinct minimal hypersurfaces constructed 
via min-max methods for a generic set of metrics, using the
Weyl law of Liokumovich--Marques--Neves \cite{LMN}.

\subsection{Overview of the proof}
The construction of a minimal hypersurface in Almgren-Pitts min-max theory proceeds by considering a sequence of sweepouts $\{ \Phi_i(x) \}$
with the supremum of the mass $\sup_x \bM(\Phi_i(x)) \rightarrow W(M)$ as $i \rightarrow \infty$.
It is then proved that we can find a subsequence $\{i_k\}$
and $\{\Phi_{i_k}(x_k) \}$ with mass tending
to $W$, so that $|\Phi_{i_k}|(x_k)$ converges to some $V \in \cR$.

We outline the proof of Theorem \ref{thm:nm+index_bound}.
For the sake of simplicity, let's focus on the non-cancellation case, i.e., when all multiplicities of $V$ are one (in the case of cancellation we must argue slightly differently but the main strategy is the same).
The main geometric idea is to show that there cannot be two disjoint open sets $U_1,U_2$ so that $\Sigma=\supp V$ fails to be one-sided homotopy minimizing on 
the same side in both $U_1$ and $U_2$. 
This property is reminiscent of (but different from) 
almost minimizing 
property introduced by Pitts to prove regularity
of min-max minimal hypersurfaces.

Granted this fact, it is easy to deduce the bound \eqref{e:bound1}. For example, if $\Index(\Sigma) = 1$ and $\hnm(\Sigma) = \{p\}$, then we can localize the index in some $U$ disjoint from $p$. Because $\Sigma$ is unstable in $U$, we can find area decreasing homotopies to both sides there, and we can also find $B_r(p)$ disjoint from $U$ with area decreasing homotopies (by definition). This contradicts the above fact. 

As such, we want to show the one-sided homotopy minimizing property in pairs by using the fact that $V$ is a min-max minimal hypersurface. However, this leads us to a major difficulty. Indeed, the approximating currents $\Phi_{i_k}(x_k)$ might cross $\Sigma$ many times, making it difficult to glue in one-sided homotopies to push down the mass. 

At a technical level, the main tool used in this paper is that it is possible to 
 simplify the one-parameter case of min-max 
 theory by constructing a nested optimal sweepout $\Phi(x)$ with $\sup \bM(\Phi(x)) = W$.
This allows us to work with one
sweepout $\Phi(x)$
instead of a sequence of sweepouts. The nested property allows us to directly ``glue in'' the one-sided homotopies to push down the mass.

The existence of a nested optimal
sweepout follows from a monotonization
technique from \cite{CL}.
There Chambers and Liokumovich proved that
each sweepout
$\Phi_i(x)$ can be replaced by a nested sweepout 
$\Psi_i(x)$ with $\sup \bM(\Psi_i(x)) \leq \sup
\bM(\Psi_i(x)) + \frac{1}{i}$. ``Nested'' here means 
that $\Psi_i(x) = \partial \Om(x)$ for a family of  open sets with $\Om(x) \subset \Om(y)$ if $x<y$. The proof
used ideas of Chambers and Rotman \cite{CR} on existence of monotone homotopies of closed curves on surfaces.

After we reparametrize $\Psi_i(x)$
by the volume swept out we obtain a sequence of families that is uniformly Lipschitz in flat topology. By Arzel\`a--Ascoli a subsequence will converge to an optimal sweepout.

In the Almgren--Pitts theory, a ``pull-tight'' procedure is used to find a varifold achieving the width with good properties. We can apply this procedure to our sweepout to deduce that one of the critical varifolds in the sweepout is a smooth (up to small singular set) minimal hypersurface. We would then like to prove that this hypersurface satisfies \eqref{e:bound1}. However, this poses another issue, namely that it could a priori be possible to push the mass of the sweepout near this critical value down, while not decreasing the global width. This could lead to an infinite sequence of ``pushing down'' operations that could create extra critical points not dealt with previously.

As such, our second main technical tool is that we construct an 
optimal sweepout $\Phi(x)$ with the following special property. 
For every point $x_0$ 
with 
$\sup_{x\in U} \bM(\Phi(x)) = W$
for every neighbourhood $U$ of $x_0$, 
there \emph{does not} exist an open interval $I \ni x_0$
and a family 
$\{\Phi'(x)\}$ that coincides with $\Phi$ for $x\notin I$ and satisfying $\sup_{x \in I'} \bM(\Phi'(x)) < W$ for every closed interval $I' \subset I$.
(In fact, we will prove a somewhat stronger property
that holds for open, half-open and closed intervals $I$).
In other words, we can not make a
small ``dip'' in the graph of $\bM(x)$,
pushing it below $W$
in the neighbourhood of $x_0$.

This property of $\Phi$ allows us to easily prove the ``homotopy minimizing to one-side'' property of $V$ discussed above (and thus Theorem \ref{thm:nm+index_bound}). A diagram of the procedure to prove this can be found in Figure \ref{fig:two-homotopies-push-down}. 

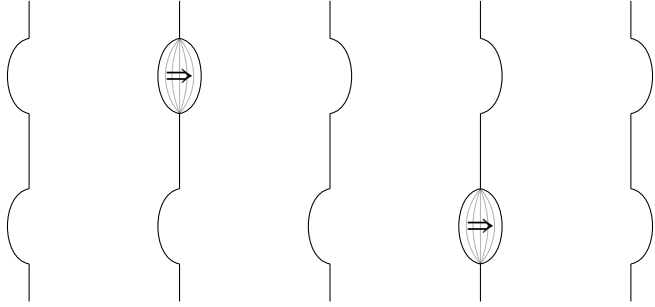
\begin{figure}
\begin{tikzpicture}

\begin{scope}[shift={(-2,0)}]
\draw (0,0) -- (0,.5);
\draw (0,-1) to [bend left = 80]  (0,0);
\draw (0,-2) -- (0,-1);
\draw (0,-3) to [bend left = 80]  (0,-2);
\draw (0,-3.5) -- (0,-3);
\end{scope}

\begin{scope}
\draw (0,0) -- (0,.5);
\draw (0,-1) to [bend left = 80]  (0,0);
\draw [opacity=.3] (0,-1) to [bend left = 40]  (0,0);
\draw [opacity=.3] (0,-1) to [bend left = 20]  (0,0);
\draw [opacity=.3]  (0,-1) to [bend left = 0]  (0,0);
\draw [opacity=.3] (0,-1) to [bend left = -20]  (0,0);
\draw [opacity=.3] (0,-1) to [bend left = -40]  (0,0);
\node at (0,-.5) {$\Rightarrow$};
\draw (0,-1) to [bend left = -80]  (0,0);
\draw (0,-2) -- (0,-1);
\draw (0,-3) to [bend left = 80]  (0,-2);
\draw (0,-3.5) -- (0,-3);
\end{scope}

\begin{scope}[shift={(2,0)}]
\draw (0,0) -- (0,.5);
\draw (0,-1) to [bend left = -80]  (0,0);
\draw (0,-2) -- (0,-1);
\draw (0,-3) to [bend left = 80]  (0,-2);
\draw (0,-3.5) -- (0,-3);
\end{scope}

\begin{scope}[shift={(4,0)}]
\draw (0,0) -- (0,.5);
\draw (0,-1) to [bend left = -80]  (0,0);
\draw (0,-2) -- (0,-1);
\draw (0,-3) to [bend left = 80]  (0,-2);
\draw [opacity=.3] (0,-3) to [bend left = 40]  (0,-2);
\draw [opacity=.3] (0,-3) to [bend left = 20]  (0,-2);
\draw [opacity=.3] (0,-3) to [bend left = 0]  (0,-2);
\draw [opacity=.3] (0,-3) to [bend left = -20]  (0,-2);
\draw [opacity=.3] (0,-3) to [bend left = -40]  (0,-2);
\draw (0,-3) to [bend left = -80]  (0,-2);
\node at (0,-2.5) {$\Rightarrow$};
\draw (0,-3.5) -- (0,-3);
\end{scope}

\begin{scope}[shift={(6,0)}]
\draw (0,0) -- (0,.5);
\draw (0,-1) to [bend left = -80]  (0,0);
\draw (0,-2) -- (0,-1);
\draw (0,-3) to [bend left = -80]  (0,-2);
\draw (0,-3.5) -- (0,-3);
\end{scope}
\end{tikzpicture}
\caption{A diagram of how to use two disjoint regions that are not one-sided homotopy minimizing to either side to push down the mass. The ``bumps'' have smaller mass than the original surface and are homotopic along one-sided homotopies that never increase mass. By doing one homotopy fully and then the other, we can see that the mass is always strictly decreased in this range. This can then be glued into a nested sweepout without increasing mass of the adjacent leaves. }\label{fig:two-homotopies-push-down}
\end{figure}


\medskip

It thus remains to explain how we perturb remaining singularities. Hardt-Simon (cf. \cite{HS}) proved that one sided perturbations of area minimizing cones with isolated singularities are smooth, and Liu extended this result to one sided stationary minimizers (cf. \cite{Liu}). The proofs of the bounds in Theorems \ref{thm:generic_bound_ricci}, \ref{thm:generic_bound}, and \ref{thm:generic_stratum} are then obtained by combining a weak generalization of White \cite{W}, where ``stability'' is replaced by the stronger hypothesis that the surface is homotopic minimizing to one side, with a simple surgery procedure in an annulus around the singularity, to show that singular points with regular tangent cones in $\sing(V) \setminus \hnm(V)$ are not generic. As such, the main results follow from \eqref{e:bound1} when combined with a result of Simon on uniqueness of the blow up at certain singularities \cite{Simon}.  The rest of the theorem follows from the index lower bound in manifolds with positive Ricci curvature.

\subsection{Organization of the paper} The paper is divided into four sections. The first section contains the basic definitions and some useful geometric tools. The second section introduces the notion of non-excessive optimal sweepouts, which is a key idea in the present work. The third section proves Theorem \ref{thm:nm+index_bound}, while the last section is dedicated to the proofs of  Theorems \ref{thm:generic_bound} and \ref{thm:generic_stratum}.

\subsection{Remark} Since the first version of this paper appeared, Li--Wang have shown in a remarkable work \cite{LiWang} that the results of this paper can be combined with \cite{Wang2020} and some further geometric arguments to prove a version of Theorem \ref{thm:generic_bound_ricci} for generic metrics without any curvature assumption.

\subsection{Acknowledgements} We are grateful to Zhihan Wang for pointing out an error in the original version of this paper and for pointing out an interesting alternative approach to prove Theorem \ref{thm:generic_bound_ricci}.
We are also grateful to Zhichao Wang for carefully reading an earlier version of this paper and making several helpful suggestions.
We would like to thank the anonymous referees for numerous corrections. 

O.C. was partially supported by a Terman Fellowship, a Sloan Fellowship, and an
NSF grant DMS-1811059/2016403. Y. L. was partially supported by NSERC Discovery grant. L. S. acknowledges the support of the NSF grant DMS-$1951070$.  

\section{Geometric preliminaries}
 
In this section we introduce the main notations, we prove the existence of optimal nested sweepouts and we recall some useful geometric tools.

\subsection{Notations} Let $(M^{n+1},g)$ be a closed Riemannian manifold. By scaling, it suffices to consider $\Vol(M,g) = 1$, which we will always assume below. We use $\mathcal{Z}_{n}(M; \mathbb{Z}_2)$ to denote
the space of mod 2 flat cycles in $M$. The topology on the space $\mathcal{Z}_{n}(M; \mathbb{Z}_2)$ is induced by the usual flat norm $\cF$.

We will make extensive use of the notion of \emph{Caccioppoli set}. A measurable set $E\subset M$ is Caccioppoli if 
\[
\Per(E):=\sup\left\{ \int_M \chi_E\,{\rm div} \omega\,:\, \omega\in X(M)\,,\,\,\|\omega\|_{C^0}\leq 1  \right\}<\infty\,,
\]
where $\chi_E$ denotes the indicator function of $E$. By De Giorgi's strutcure theorem we have that the distributional derivative $D\chi_E$ (which is is a Radon measure) of a set of finite perimeter $E$ is given by $D\chi_E=\nu_E\, \cH^{n}\res \partial^* E$, where $\partial^* E$ is the reduced boundary of $E$, which is a $n$-rectifiable set, and $\nu_E$ is the normal direction to $\partial^*E$ pointing outside $E$ defined $\cH^{n}$-a.e.. This allows us to identify $D\chi_E$ with an element of $\mathcal{Z}_{n}(M; \mathbb{Z}_2)$, which we will abuse notation and denote by
\[
\partial E := \nu_E\, \cH^{n}\res \partial^* E\,.
\]

In particular, with this identification we have
$$
\bM(\partial E)=\Per(E)\qquad\text{and}\qquad \cF(\partial E, \partial F)=\| \chi_E -\chi_F\|_{L^1}=\Vol(E\Delta F)\,,
$$
where $E\Delta F$ denotes the symmetric difference between two sets. As usual, the perimeter of $E$ in an open set $U$, denoted by $\Per(E\,|\,U)$, is the total variation of $D\chi_E$ in the set $U$.

We let $\cV$ denote the set of varifolds of $(M^{n+1},g)$. Given a cycle $\Gamma\in \mathcal{Z}_{n}(M; \mathbb{Z}_2)$, we will write $|\Gamma|$ for the associated integral varifold. In particular if $\Gamma=\partial \Omega$, then 
$$
|\Gamma|=\cH^{n}\res \partial^*\Omega\,\otimes\, \delta_{T\partial^*\Omega}\qquad \text{and}\qquad
\mu_{\Gamma}:=\cH^{n}\res \partial^*\Omega
$$ 
is the total variation measure of the measure $D\chi_\Omega$. Finally, given a set $\Sigma$ regular outside a set of dimension $n-7$, we will denote with $|\Sigma|$ the associated integral varifold.

\subsection{Optimal nested sweepouts} We start by recalling the notion of sweepouts that we will use in this paper.

\begin{definition}[Sweepout]
	A sweepout of $M$ is a map $\Phi: [0,1] \rightarrow \mathcal{Z}_{n}(M; \mathbb{Z}_2)$ continuous in $\cF$-topology such that $\Phi(x) = \partial \Omega(x)$, where $\{\Omega(x) \,:\, x\in [0,1]\}$ is a family of Caccioppoli sets with $\Omega(0)$ the $0$-cycle and $\Omega(1) = M$.  We will denote with $\cS$ the collection of all such sweepouts.
	Moreover, we define the width $W$ to be
\[
W = \inf_{\Phi \in \cS} \,\sup_{x\in [0,1]} \bM(\Phi(x)) \,.
\]
	It is a consequence of the isoperimetric inequality that $W >0$ \cite[Proposition 0.5]{DL}.
\end{definition}

We will switch freely between the equivalent notation $\bM(\Phi(x))$ and $\Per(\Omega(x))$. We now  introduce the notion of optimal nested sweepouts and prove their existence.

\begin{definition}[Optimal nested volume parametrized (ONVP) sweepout]
	A sweepout $\{ \Phi(x) = \partial \Omega(x) \,:\, x\in[0,1]\}$ is called
	\begin{itemize}
		\item \textit{optimal} if $\sup_{x\in [0,1]} \bM(\Phi(x)) = W$;
		\item \textit{nested} if $\Omega(x_1) \subset \Omega(x_2)$, for all $0\leq x_1 \leq x_2\leq 1$;
		\item \textit{volume parametrized} if $\Vol(\Omega(x)) = x$, for every $x\in [0,1]$ (recall that we have assumed $\Vol(M,g) = 1$). 
	\end{itemize}
\end{definition}

Nested volume parametrized sweepouts enjoy nice compactness properties.

\begin{lemma}[Compactness for nested volume parametrized sweepouts] \label{nested sequence}
	Let $\left( \Phi_i\right)_i$ be a sequence of nested volume-parametrized sweepouts
	with mass uniformly bounded, that is
	\begin{equation}\label{e:mass_bounds}
     \sup_{i\in \mathbb N} \sup_{x\in [0, 1]} \bM(\Phi_i(x))\leq M<\infty\,.
	\end{equation}
	Then there exists a subsequence $\left(\Phi_{i_k}\right)_k$ converging uniformly to a nested volume parametrized sweepout $\Psi$ such that
	\begin{equation}\label{e:sup_optimal}
		\sup_x \bM(\Psi(x)) \leq \liminf_k \left(\sup_x \bM( \Phi_{i_k}(x) )\right).
	\end{equation}
\end{lemma}

\begin{proof} The sequence of continuous functions $\Phi_i\colon [0,1] \to \mathcal{Z}_{n-1}(M; \mathbb{Z}_2)$ is uniformly Lispchitz continuous, since for every $0\leq x<y\leq 1$ we have
	$$
	\cF(\Phi_i(x), \Phi_i(y))\leq \Vol( \Omega_i(y)\setminus\Omega_i(x)) = \Vol(\Omega_i(y))-\Vol(\Omega_i(x))=y-x
	$$
	and $\Phi_i(0)=\emptyset$ for every $i$, so by Arzel\`a--Ascoli Theorem there exists a subsequence $\Phi_{i_k}$ and a nested volume parametrized sweepout $\Psi\colon [0,1] \to \mathcal{Z}_{n}(M; \mathbb{Z}_2)$ such that $\Phi_{i_k}$ converges uniformly to $\Psi$. Then \eqref{e:sup_optimal} follows from \eqref{e:mass_bounds} and the lower semi-continuity of $\bM$ with respect to the flat topology.
\end{proof}

Optimal nested volume parametrized sweepouts exist.

\begin{theorem}[Existence of (ONVP) sweepouts] \label{nested exists}
	For any closed Riemannian manifold $(M,g)$ there exists an optimal 
	nested volume-parametrized sweepout.
\end{theorem}

\begin{proof}
	Let $ \{\Psi_i\}_i$ be a min-max sequence of sweepouts with $\lim_{i \rightarrow \infty }\sup_x \bM(\Psi_i(x)) = W$.
	By \cite[Theorem 1.4]{CL}
	we can replace $\{\Psi_i(x)\}_i$ by a  sequence of nested sweepouts $\{\Phi_i\}_i$, such that $\Phi_i(x) = f_i^{-1}(x)$ for
	a surjective Morse function $f_i: M \rightarrow [0,1]$ and $\lim_{i \rightarrow \infty }\sup_x \bM(\Phi_i(x)) = W$. 
	Let $\phi_i(x) = \Vol (f_i^{-1}([0,x]))$.
	Since $f_i$ is a Morse function we have that $\phi_i: [0,1] \rightarrow [0,1]$ is a continuous strictly increasing function. Then $\left( \Phi_i\circ \phi_i^{-1}\right)_i$ is a sequence
	of nested volume-parametrized sweepouts.
	By Lemma \ref{nested sequence}, a subsequence of $\{ \Phi_i\circ \phi_i^{-1}\}_i$ converges 
	to a (ONVP) sweepout.
\end{proof}

Finally we recall the definition of critical set for a sweepout. 

\begin{definition}[Critical set]
	Given a sweepout $\Phi$, we define
	$$
	M(x) = \limsup_{r \rightarrow 0} \{\bM(\Phi(y)): |y-x| < r\}.
	$$
	If $\Phi$ is an optimal sweepout, we define the \emph{critical domain of $\Phi$} to be the set 
	$$
	\mm(\Phi)=\{ x \in [0,1]: M(x) =W \}\,.
	$$
	We will say that a sequence $x_i \rightarrow x \in \mm(\Phi)$ is a \emph{min-max sequence} if $|\Phi|(x_i)$ converges in the varifold sense to a varifold $V$ of mass $W$, i.e. $|V|(M)=W$. We denote the set of such varifolds by $\C(\Psi)$. 
\end{definition}

In fact, it is convenient to refine this definition somewhat. 

\begin{definition}[Left and right critical set]
Given a sweepout $\Phi$, we say that $x\in \mm_L(\Phi)$ if there are $x_i\nearrow x$ with 
\[
\bM(\Psi(x_i)) \to W
\]
and similarly, $x\in \mm_R(\Phi)$ if there are $x_i\searrow x$ with 
\[
\bM(\Psi(x_i)) \to W. 
\]
\end{definition}

Note that $\mm(\Phi) = \mm_L(\Phi) \cup \mm_R(\Phi)$ (and $\mm_L(\Phi) \cap \mm_R(\Phi)$ need not necessarily be empty).

\begin{definition}[Varifolds of optimal regularity]
For an open set $U \subset M$ we say that a varifold $V$ is in $\cR(U)$ if 
\[
V\res U = \sum_{k=i}^K \kappa_i \,|\Sigma_i|\,,
\]
for $\Sigma_i\subset U$ embedded minimal hypersurfaces that are regular away from  a closed singular set of Hausdorff dimension $n-7$, and $\kappa_i\in\NN$ integer multiplicities. We set $\cR = \cR(M)$. 

Abusing notation, we will say that a (singular) hypersurface $\Sigma$ is in $\cR(U)$ if the associated varifold satisfies $|\Sigma|\in\cR(U)$. 
\end{definition}

As we will discuss later, Almgren--Pitts theory \cite{Pitts} implies that given an optimal sweepout $\Psi$, there is $V\in\C(\Psi)$ with $V \in \cR$. 

\subsection{(Homotopic) one sided minimizers} 
The notions of one sided minimizers and homotopic minimizers will be useful when dealing with our deformations theorems, as we will see in Section \ref{ss:deformation_thm}. 

Let $U \subset M^{n+1}$ be an open set. Given two Caccioppoli sets $E\subset \Omega$ with $E \Delta \Omega \subset U$ and $\eps \geq 0$, we denote the inner families of deformations between $E$ and $\Omega$ in $U$ which do not increase the volume more than $\eps$ by
\[
\cI_\eps(\Omega, E\,|\,U):=\left\{ \{\Omega(t)\}_{t \in [0,1]} :  \begin{gathered} \, \Omega(0) = E,\, \Omega(1) = \Omega, (\Omega(t)\Delta \Omega) \setminus U =\emptyset,\\ \Omega(t_1) \subset \Omega(t_2) \text{ for } t_1 < t_2, \\
\Per(\Omega_t)  \leq \Per(\Omega) + \eps\ \end{gathered}   \right\}\,, 
\]
and analogously for $\Omega\subset E$ with $E \Delta \Omega \subset U$, we define the outer families of deformations between $\Omega$ and $E$ in $U$ which do not increase the volume more than $\eps$ by
\[
\cO_\eps(\Omega, E\,|\, U):=\left\{\{\Omega(t)\}_{t \in [0,1]}:\begin{gathered} \Omega(0) =\Omega,\, E = \Omega(1), (\Omega(t)\Delta \Omega) \setminus U =\emptyset,\\ \Omega(t_1) \subset \Omega(t_2) \text{ for } t_1 < t_2  \\ \Per(\Omega_t) \leq \Per(\Omega) + \eps\end{gathered} \right\}\,. 
\]
Moreover, given $\Omega$ and $\eps\geq 0$, we denote the collections of inner and outer Caccioppoli sets that can be reached by an inner or outer family of deformations by
\begin{gather*}
    \cI_\eps(\Omega\,|\, U)=\{E\subset \Omega\,: E \Delta \Omega \subset U,\, \cI_\eps(\Omega, E\,|\,U)\neq \emptyset\},\\
    \cO_\eps(\Omega\,|\, U)=\{E\supset \Omega\,: E \Delta \Omega \subset U, \, \cO_\eps(\Omega, E\,|\,U)\neq \emptyset\}\,.
\end{gather*}
In both definitions, if we do not include an $\eps$-subscript, it should be understood that we are taking $\eps=0$; this will happen most of the time below, but we will crucially rely on the definition with $\eps>0$ to obtain \emph{regular} homotopic minimizers in certain situations.

\begin{definition}[Homotopic inner and outer minimizers] \label{homotopy minimizer}
	Given a Caccioppoli set $\Omega$ we say that a Caccioppoli set $L(\Omega\,|\, U)\in \cI(\Omega\,|\, U)$ is a \emph{homotopic inner minimizer for $\Omega$ in $U$}, if
	\begin{enumerate}
		\item $\Per(L(\Omega\,|\, U)\,|\, U)\leq \Per(\Omega'\,|\,U)$, for every $\Omega'\in \cI(\Omega\,|\, U)$ and
        \item if $E\in \cI(\Omega\,|\, U)$ satisfies (1) and $ L(\Omega\,|\, U) \subset E$ then $E= L(\Omega\,|\, U)$.
	\end{enumerate}
	Similarly, define $R(\Omega\,|\, U)\in \cO(\Omega\,|\, U)$ to be a \emph{homotopic outer minimizer for $\Omega$ in $U$}, if
	\begin{enumerate}
		\item $\Per( R(\Omega\,|\, U)\,|\,U)\leq \Per(\Omega'\,|\,U)$, for every $\Omega'\in \cI(\Omega\,|\, U)$;
\item if $E\in \cI(\Omega\,|\, U)$ satisfies (1) and $E \subset R(\Omega\,|\, U)$ then $E= R(\Omega\,|\, U)$.
	\end{enumerate}
	We say that a Caccioppoli set $\Omega$ is an \emph{inner (resp. outer) homotopic minimizer in $U$} if $\Omega$ is a homotopic inner (resp.\ outer) minimizer relative to itself. 
\end{definition}

It is easy to see that inner and outer homotopic minimizers for a fixed set $\Omega$ always exist.

\begin{lemma}[Existence of homotopic minimizers] \label{l:existence_homotopic}
	For any Caccioppoli set $\Omega$ and open set $U$	we can find a homotopic inner (resp.\ outer) minimizer $L(\Omega\,|\, U)$ (resp.\ $R(\Omega\,|\, U)$) for $\Omega$ in $U$. 
\end{lemma}

\begin{proof} We consider only the case of inner minimizers as the outer minimizers are handled identically. 

This is once again an application of Arzel\`a--Ascoli theorem. Indeed, notice that $\cI(\Omega\,|\, U) \not = \emptyset$, since $\Omega \in \cI(\Omega\,|\, U)$, so we can consider a minimizing sequence $(E_j)_j$, that is 
\[
	\lim_{j}\Per( E_j\,|\,U)=\inf\{ \Per(E\,|\,U)\,:\, E\in  \cI(\Omega\,|\, U)
	\}
\]
	and let $\{E_j(x): x\in [0,1]\} \in \cI(\Omega,E_j \,|\, U)$ be the corresponding inner volume non increasing sweepout between $E_j$ and $\Omega$. We can assume that it is volume parametrized (being nested). Moreover $\Per(E_j(x)\,|\,U)$ is uniformly bounded by $\Per(\Omega\,|\,U)$, so by Arzel\`a--Ascoli there is a subsequence converging to $\{E_\infty(x)\} \in \cI(\Omega, E_\infty\,|\, U)$, with $E_\infty$ satisfying the desired minimality property by lower semi-continuity of the perimeter. 
	
	Finally, again by Arzel\`a--Ascoli, we can find $L(\Omega\,|\, U)\subset \Omega$ in the set of minimizers, which infimizes the flat distance to $\partial \Omega$, and so satisfies condition (2) (otherwise there would be a competitor closer to $\Omega$ in flat norm).
\end{proof}

We recall the definition of one-sided minimizers, which will be useful in the sequel when we perform cut and paste arguments. 

\begin{definition}[One sided minimizers]
	Let $E$ be a Caccioppoli set. We say that $E$ is \emph{locally one-sided inner (resp. outer) area-minimizing} in $U$ if for every $A\Subset U$ and $V$ with $V\Delta E \subset A$, we have
\[
	\Per( E\,|\, A) \leq \Per (V \,|\, A)
\]
	whenever $V\subset E$ (resp.\ $E\subset V$). We say that $E$ is \emph{strictly locally one-sided inner (resp.\ outer) area-minimizing} if the inequality holds strictly except when $E=V$ as Caccioppoli sets. 
\end{definition}

We show that homotopic minimizers are in fact strict one sided minimizers into the region they sweep out. 

\begin{lemma}[Homotopic minimizers are one sided minimizers in the swept out region] \label{strict minimizer}
	Suppose $L(\Omega\,|\, U)$ is an homotopic inner (resp. outer) minimizer for $\Omega$ in $U$. Then $L(\Omega\,|\, U)$ (resp. $R(\Omega\,|\, U)$) is strict locally outer (resp. inner) one-sided minimizing in $U\cap \Omega$ (resp. $U\setminus \Omega$) .
\end{lemma}

\begin{proof}
	We consider homotopic inner minimizers; the case of outer minimizers is similar. 
	
	If $L(\Omega\,|\, U)$ is not a strict outer minimizer in $U\cap \Omega$ then there is $V'$ with $L(\Omega\,|\, U)\subset V'$ and $L(\Omega\,|\, U) \Delta V' \subset A \Subset U$ and
	\[
	\Per(V'\,|\,A) \leq P(L(\Omega\,|\, U)\,|\,A).
	\]
	We can minimize perimeter in $A$ among all such $V'$ to find $V$. Namely,
	\begin{equation}\label{e:confusing1}
	\Per( V \,|\,A )\leq \Per(W\,|\,A)  
	\end{equation} 
		for all $W$ with $W\Delta V\subset A\setminus L(\Omega\,|\, U)$. Since $L(\Omega\,|\, U) \in \cI(\Omega\,|\, U)$, there is $\{U(x)\,:\,x \in [0,1]\} \in \cI(\Omega, L(\Omega\,|\,U)\,|\,U)$. Set $\Omega(x) = U(x) \cup V$. Since $V$ satisfies \eqref{e:confusing1},
we have that 
	$$
	\Per(\Omega_t\,|\,A) \leq \Per(U_t\,|\,A).
	$$
	This implies that $\Omega(1)=V$ satisfies (1) of Definition \ref{homotopy minimizer} and $V \Delta L(\Omega\,|\,U)\subset A \setminus L(\Omega\,|\, U)$, therefore by (2) of Definition \ref{homotopy minimizer}, it follows that $V= L(\Omega\,|\,U)$. This completes the proof. 
\end{proof}

We have the following lemma that will allow us to find bounded mass homotopies in certain situations. 
\begin{lemma}[Interpolation lemma] \label{l:close in flat}
	Fix $L>0$. For every $\eps>0$ there exists $\delta>0$, such that the following holds. If $\Om_0,\Om_1 $ are two sets of finite perimeter, such that $\Om_0 \subset \Om_1$, $\Per(\Om_i) \leq L$, $i=0,1$, and $\Vol(\Om_1 \setminus \Om_0)\leq \delta$, then there exists a nested $\cF$-continuous family $\{\partial\Om_t \}_{t \in [0,1]}$ with 
\[
\Per(\Om_t)\leq \max\{\Per(\Om_0),\Per(\Om_1) \}+\eps
\]
for all $t\in[0,1]$
\end{lemma}
\begin{proof}
Let $\Om$ be a Caccioppoli set that minimizes
perimeter among sets $\Om'$
with $\Om_0 \subset \Om' \subset \Om_1$.

Fix $r>0$ such that for every $x \in M$ the ball $B(x,2r)$
is $2$-bi-Lipschitz diffeomorphic to
the Euclidean ball of radius $2r$.
Let $\{B(x_i,r) \}_{i=1}^N$ be a collection of balls covering $M$. By coarea inequality
we can find a radius $r_i \in [r,2r]$,
so that $\bM(\partial B(x_i,r_i) \cap \Om \setminus \Om_0) \leq \frac{\delta}{r}$.

Let $U_1 = B(x_1,r_1) \cap
\Om \setminus \Om_0$.
By a result of Falconer (see \cite{Falconer},
\cite[Appendix 6]{Guth}) there exists a family of 
hypersurfaces sweeping out $U_1$ of area bounded by
$c(n) \delta^{\frac{n}{n+1}}$. It follows (see
\cite[Lemma 5.3]{CL}) that
there exists a nested family 
$\{ \Xi^1(t)\} $ of Caccioppoli sets with
$\Xi^1(0) = \Om_0$ and $\Xi^1(1) = \Om_0 \cup U_1$
and satisfying
$$\Per(\Xi^1(t)) \leq \Per(\Om_0)+2c(n)\delta^{\frac{n}{n+1}} \, . $$
Let $\Om^1 =  \Om_0 \cup U_1$. Observe, that
the minimality of $\Om$ implies that
$$\Per(\Om^1) \leq \Per(\Om_0) + \frac{2\delta}{r} $$

Inductively, we define $\Om^k= \Om^{k-1} \cup U_k$ and
$U_k = B(x_k,r_k) \cap
\Om \setminus \Om^{k-1}$. As above we can construct a nested homotopy of Caccioppoli sets $\Xi^k(t)$ from $\Om^{k-1}$ to $\Om^k$, satisfying
$$\Per(\Xi^k(t)) \leq \Per(\Om_0)+2c(n)\delta^{\frac{n}{n+1}} 
+ \frac{2N\delta}{r}$$

We choose $\delta>0$ so small that $\Per(\Xi^k(t))<  \Per(\Om_0) + \eps$.
It follows then that we have obtained a homotopy from $\Om_0$ to $\Om$ satisfying the
desired perimeter bound. Similarly,
we construct a homotopy from $\Om$ to $\Om_1$.
\end{proof}




Finally, we have the following result. Recall that White \cite{W} proves that strictly stable  \emph{smooth} minimal hypersurfaces are locally area-minimizing. A generalization of such a result to the case of hypersurfaces with singularities (i.e., elements of $\cR$) would be very interesting. The following (weaker) result will suffice for our needs; it can be seen as a result along these lines, except ``stability'' is replaced by a stronger hypothesis: the surface is homotopic minimizing to one side.\footnote{Note that one certainly needs a condition on the singularities rather than just a condition on the regular part like strict stability, since as we show in Proposition \ref{p:def_thm}, the existence of (regular) non-minimizing tangent cones implies that the hypersurface is not homotopic minimizing irrespective of any stability condition that might hold on the regular part.}

 \begin{proposition}[Comparing the notions of minimizing vs.\ homotopic minimizing for minimal surfaces]\label{prop:min-vs-htpy-min}
Suppose that $\Omega$ is a Caccioppoli set and for some strictly convex open set $U \subset M$ with smooth boundary, the associated varifold $V = |\partial \Omega|$ satisfies $V \in \cR(U)$. Assume that $\supp V \cap U$ is connected. 

Suppose that $\Omega$ is inner (resp.\ outer) homotopy minimizing in $U$. Then, at least one of the following two situations holds:
\begin{enumerate}
\item for all $p \in \supp V \cap U$, there is $\rho_0>0$ so that for $\rho<\rho_0$, $B_\rho(p) \subset U$ and $\Omega$ is inner (resp.\ outer) minimizing in $B_\rho(p)$, or
\item there exists a sequence of Caccioppoli sets $E_i
\neq \Omega$
with $|\partial E_i| \in \cR(U)$ so that $E \Delta \Omega \subset \Omega \cap U$ (resp.\ $\Omega^c \cap U$), $|\partial E_i|$ has stable regular part, and $\partial E_i \to \partial \Omega$ in the flat norm.  
\end{enumerate}
\end{proposition}

\begin{remark}
It is interesting to ask if the second possibility occurs; it seems possible that one could rule this out in the case where $V$ has regular tangent cones that are all strictly minimizing in the sense of Hardt--Simon \cite[\S 3]{HS}. 
\end{remark}

\begin{proof}[Proof of Proposition \ref{prop:min-vs-htpy-min}]
We consider the ``inner'' case, as the ``outer'' case is similar. Let $E^\eps \in \cI_\eps(\Omega\, |\, U)$ minimize perimeter among all sets in $\cI_\eps(\Omega\,|\, U)$ (as usual, the existence of $E^\eps$ follows from Arzel\`a--Ascoli). We claim that $E^\eps$ is area-minimizing to the inside of $\Omega$ in sufficiently small balls. 

More precisely, for $r>0$ sufficiently small, suppose there was a Caccioppoli set $E'$ so that $E'\Delta E^\eps \subset B_r(p) \cap U \cap \Omega$ and $\Per(E'\,|\,U) < \Per(E^\eps\,|\,U)$. As long as $r$ was chosen sufficiently small, Lemma \ref{l:close in flat} guarantees
that $E' \in \cI_\eps(\Omega\, |\,U)$. This is a contradiction. 

Now, consider $p \in \reg V \cap U$. We note that $E^\eps$ is almost minimizing (with no constraint coming from $\Omega$) in the sense of \cite{Tam}, and thus has $C^{1,\alpha}$ boundary in $B_r(p) \cap U$,  thanks to standard results on the obstacle problem; see \cite[\S 1.9, \S1.14(iv)]{Tam}. As such, away from $\sing V$ (which has Hausdorff dimension at most $n-7$) we can thus conclude that $\partial^*E^\eps$ is regular, stationary and stable.\footnote{Cf.\ the proof of \cite[Proposition 2.1]{Liu} for the proof of stability.} A capacity argument then implies that $|\partial E^\eps| \in \cR(U)$ and $\partial^*E^\eps$ is stable. Therefore, the maximum principle for (possibly singular) hypersurfaces \cite{Ilm} implies that either $E^\eps = \Omega$ or $\partial^*E^\eps \cap \supp V = \emptyset$. In the first case, we can conclude that $\Omega$ is inner minimizing in small balls (since $E^\eps$ is). 

We can thus assume that the latter possibility holds for all $\eps>0$ sufficiently small. Taking $\eps_j\to 0$, there is $E \in \cI_0(\Omega | U)$ so that $E^{\eps_j}\to E$ with respect to the flat norm. If $E=\Omega$, then the second possibility in the conclusion of the proposition holds for $E_j = E^{\eps_j}$. 

The final case to consider is $E\neq \Omega$. By curvature estimates for stable minimal hypersurfaces \cite{SS}, $|\partial E| \in \cR(\Omega)$ and thus $\partial^*E \cap \supp V = \emptyset$ again by the maximum principle.

We know that $\Per(E_i| U) \leq \Per(\Omega| U)$, so in the limit we get
$\Per(E| U) \leq \Per(\Omega| U)$
By Arzel\`a-Ascoli nested homotopies from $\Omega$ to $E_i$ will converge to
a nested homotopy $E(t)$ from $\Omega$ to $E$ that does not increase volume.
By the inner homotopy minimizing property of $\Omega$ we have
\[
\Per(E\,|\,U) = \Per(\Omega\, |\, U).
\]

Suppose we minimize perimeter among Caccioppoli sets $A'$ sandwiched between $E$ and $\Omega$, $E \subset A' \subset \Omega$. We claim that the
minimizer $A$ has perimeter equal to that of $\Omega$. Indeed, if  $A$ has strictly
smaller perimeter, then family $E(t) \cup A$
is an area decreasing nested homotopy between $\Omega$ and $A$, contradicting that
$\Omega$ is inner homotopic minimizing.

We thus see that $\Omega$ is minimizing in $\Omega \cap E^c \cap U$, which implies that it is inner minimizing in small balls, as asserted. 
\end{proof}

\section{Non-excessive sweepouts}\label{sec:min-max}

In this section we introduce the concept of excessive intervals and excessive points for a sweepout and prove that there is a sweepout, such that every point in the critical domain is not excessive.

\begin{definition}[Excessive points and intervals] \label{def:excessive_interval}
	Suppose $\{\Phi(x)=\partial \Omega(x)\}$ is a sweepout. Given a connected interval $I$ (we allow $I$ to be open, closed, or half-open)
	we will say that $\{\Phi^I(x) = \partial \Om^I(x)\}_{x\in \bar I}$ is an \emph{$I$-replacement family for $\Phi$} if
	$\Om^I(a) = \Om(a)$, $\Om^I(b)=\Om(b)$ and for all $x \in I$, 
\[
\limsup_{I \ni y\to x} \bM(\Phi^I(y)) < W.
\]
We say that a connected interval $I$ is an \emph{excessive interval for $\Phi$} if there is an $I$-replacement family for  $\Phi$. We say that a point $x$ is \emph{left (resp.\ right) excessive for $\Phi$} if there is an excessive interval $I$ for $\Phi$ so that $(x-\eps,x]\subset I$ (resp.\ $[x,x+\eps)\subset I$) for some $\eps>0$. 
\end{definition}

 The goal of this section is to prove the following result.

\begin{theorem}[Existence of non-excessive min-max hypersurface]\label{c:non-excessive_minmax}
There exists a (ONVP) sweepout $\Psi$ such that every $x\in \mm_L(\Psi)$ is not left excessive and every $x \in \mm_R(\Psi)$ is not right excessive. 
\end{theorem}

\subsection{Preliminary results} We establish several results that will be used in the proof of Theorem \ref{c:non-excessive_minmax}.

\begin{lemma}[Extension lemma I]\label{l:union-excessive}
If $I,J$ are excessive for $\Phi$ and $I \cap J \not = \emptyset$, then $I \cup J$ is excessive for $\Phi$. 
\end{lemma}
\begin{proof}
Let $\{\partial \Om^I(x) \}_{x \in I}$
and $\{\partial \Om^J(x) \}_{x \in J}$
be $I$ and $J$ replacement families
for $\Phi$.

Let $a_1 = \inf\{x \in I\} $, $a_2 = \inf\{x \in J\}$ and $b_1 = \sup\{x \in I\} $, $b_2 = \sup\{x \in J\}$.
Assume without any loss of generality that
$a_1\leq a_2$ and $b_1 \leq b_2$
and at least one of the two inequalities
is strict.

Let $K = I \cap J$; let $a,b$ denote, respectively, left and right
boundary points of $K$ and $c = \frac{a+b}{2}\in K$.
Let $\tilde{\Om}$ be a Caccioppoli set minimizing perimeter among all $\Om'$
with $\Om(a) \subset \Om' \subset \Om(b)$.
Define $\phi_1:[a_1,c] \rightarrow [a_1,b_1]$
and $\phi_2:[c,b_2] \rightarrow [a_2,b_2]$
given by $\phi_1(x)=a_1+\frac{b_1-a_1}{c-a_1}(x-a_1)$
and $\phi_2(x)=a_2 + \frac{b_2-a_2}{b_2-c}(x-c)$.
We define an $I\cup J$ replacement family for
$\Phi$ by setting
\[
\Phi^{I\cup J}(x) = \begin{cases} 
\partial (\Om^I(\phi_1(x)) \cap \tilde{\Om}) & x \in [a_1,c]\\
\partial (\Om^J(\phi_2(x)) \cup \tilde{\Om}) & x \in [c,b_2]
\end{cases}
\]
Observe that $\Phi^{I\cup J}$
is continuous since $\Phi^{I\cup J}(c) = \partial \tilde{\Om}$.
It follows from our choice of $\tilde{\Om}$ that
$\bM(\Phi^{I\cup J}(x))\leq \bM(\Phi^I(\phi_1^{-1}(x)))<W$
for $x \in I \cap (-\infty, c] $ and 
$\bM(\Phi^{I\cup J}(x))\leq \bM(\Phi^J(\phi_2^{-1}(x)))<W$
for $x \in J \cap [c, \infty) $. 
\end{proof}

\begin{lemma}[Extension lemma II]\label{l:union-excessive-repeated}
If $I$ is excessive for $\Phi$ and $J$ has $J\cap I \neq \emptyset$ and is excessive for
\[
\Psi(x) : = \begin{cases} \Phi^I(x) & x \in I \\ \Phi(x) & x \not \in I \end{cases}
\]
then $J \cup I$ is excessive for $\Phi$. 
\end{lemma}
\begin{proof}
Define an $I\cup J$-replacement family $\Phi^{I\cup J}$ for $\Phi$ by
\[
\Phi^{I\cup J}(x) = \begin{cases} 
\Phi(x) & x \in 
[0,1] \setminus (I \cup J) \\
\Phi^I(x) & x\in I \setminus J \\
\Psi^J(x) & x \in J
\end{cases}
\]
where $\Psi^J$ is a $J$-replacement family for $\Psi$. 
\end{proof}

The following is the technical core of the proof of Theorem \ref{c:non-excessive_minmax}. 

\begin{proposition}[Existence of maximal excessive intervals]\label{p:max_int}
Given an (ONVP) sweepout $\Phi$, if $\hat J$ is excessive for $\{\Phi(x)=\partial\Omega(x)\}$, then there exists an excessive interval $J \supset \hat J$ so that $J$ is maximal in the sense that if $\tilde J$ is excessive with $\tilde J \cap J \not = \emptyset$, then $\tilde J \subset J$. 
\end{proposition}

\begin{proof}
Let 
\[
\alpha : = \sup\{|\tilde J| : \tilde J \textrm{ excessive }, \hat J \cap \tilde J \not = \emptyset\}. 
\]
Choose excessive intervals $\tilde J_n$ with $\tilde J_n \cap \hat J \neq \emptyset$ and $|\tilde J_n|\to \alpha$. By Lemma \ref{l:union-excessive}, we can replace $\tilde J_n$ by $\tilde J_n \cup \hat J$, and thus assume that $\hat J \subset \tilde J_n$. In particular $\tilde J_n \cap \tilde J_m\neq \emptyset$ for all $m,n$. Using Lemma \ref{l:union-excessive} again, we can replace $\tilde J_n$ by
\[
\bigcup_{m=1}^n\tilde J_m
\]
so that the $\tilde J_n$ form an increasing sequence of excessive intervals (still with $|\tilde J_n| \to\alpha$). Note that the interior of an excessive interval is still excessive, so we can consider $J_n : = (\tilde J_n)^\circ$. Note that $|J_n| \to \alpha$ and the $J_n$ are increasing. 

We will show below that 
\[
J'  : = \bigcup_n J_n
\]
is excessive. Write $J' = (a,b)$. Granted the fact that $J'$ is excessive, we claim that one of the intervals $(a,b),(a,b],[a,b)$, or $[a,b]$ is the desired maximal excessive interval. Note that by Lemma \ref{l:union-excessive}, if $[a,b)$ and $(a,b]$ are excessive, then so is $[a,b]$, so we can choose the largest excessive interval out of these four choices and call it $J$. Suppose that $\tilde J$ is excessive with $\tilde J \cap J\not = \emptyset$. Then, $J \cup \tilde J$ is excessive by Lemma \ref{l:union-excessive} and $\hat J \subset J \cup\tilde J$. Thus, 
\[
|J \cup \tilde J| \leq \alpha,
\]
so $\tilde J \subset \bar J$ (where $\bar J$ is the closure of $J$). Now, $J \cup \tilde J$ is excessive, but strictly larger than $\tilde J$ (by assumption). This contradicts the choice of $J$ as the largest excessive interval out of $(a,b),(a,b],[a,b)$, and $[a,b]$. This shows that $J$ is maximal, as desired.

It thus remains to prove that $J' = \cup_n J_n$ is excessive for a nested sequence of open excessive intervals $J_n$. Write $J_n = (a_n,b_n)$ and set $a'_n = a_n + \frac 1n,b_n'=b_n - \frac 1n$.

Fix $i =0,1,\dots$ and assume we have real numbers $0<A_1,\dots,A_i < W$ and integers $n_i\geq i$ (with $n_1< n_2<\dots<n_i$) so that for $n\geq n_i$, there is a $J_n$-replacement $\{\Phi^n_i(x)=\partial \Omega^n_i(x)\}$ so that 
\[
\Per(\Omega^n_i(x)) \leq A_j
\]
for $x \in [a_j',b_j']$ and $1 \leq j \leq i$. (Note that for $i=0$, we can find such objects because the $J_n$ are excessive.) 

We will choose $0 < A_{i+1} < W$, and $n_{i+1}>\max\{n_i,i+1\}$ so that we can construct $J_n$-replacements $\{\Phi^n_{i+1}(x)=\partial \Omega^n_{i+1}(x)\}$ for $n\geq n_{i+1}$ with 
\[
\Per(\Omega^n_{i+1}(x)) \leq A_j
\]
for $x \in [a_j',b_j']$ and $1 \leq j \leq i+1$. Granted this, we can easily (inductively) complete the proof by passing $\Phi^{n_{i+1}}_{i+1}$ to a subsequential limit (using Arzel\`a--Ascoli).

 It is useful to introduce the following notation, used in the construction of $\Phi^n_{i+1}$. Given two nested sets of finite perimeter $V \subset W$, we
	let 
	\begin{itemize}
		\item $\cM_{V,W}$ an outermost Caccioppoli set minimizing perimeter among all the Caccioppoli sets $\Om$ with $V \subset \Om \subset W$;
		\item $\{\cV_{(V,W)}(x)\}_x$ the optimal nested homotopy from $V$ to $W$.
	\end{itemize} 
For $n \geq n_i$, we set
\[
L_n : = \cM_{\Omega(a_n),\Omega^{n_i}_i(a_{i+1}')}, \qquad U_n: = \cM_{\Omega^{n_i}_i(b_{i+1}'),\Omega(b_n)}
\]
Note that for $n \leq m$, $ L_m \subset L_n$ and $U_n \subset U_m$. Hence, $L_n$ and $U_n$ have $\cF$-limits as $n\to\infty$. For $\eps > 0$ fixed so that 
\[
\max\left\{\Per\left(\Omega^{n_{i}}(a_{i+1}')\right),\Per\left(\Omega^{n_i}_i(b_{i+1}')\right)\right\}+\eps < W,
\]
Lemma \ref{l:close in flat} thus guarantees that there is $n_{i+1}\geq i+1$ sufficiently large so that for $n\geq n_{i+1}$, 
\[
\sup_t \Per\left(\cV_{(L_n,L_{n_{i+1}})}(t)\right) < W, \qquad \sup_t \Per\left(\cV_{(U_{n_{i+1}},U_n)}(t)\right) < W. 
\]
For $n\geq n_{i+1}$, we define
\[
\tilde \Phi^n_{i+1}(x) = \begin{cases}
\partial\left(\Omega^n_i(x+1) \cap L_n\right) & x \in [a_n-1,b_n-1]\\
\partial \tilde \cV_{L_n,L_{n_{i+1}}}(x) & x \in [b_n-1,a_{n_{i+1}}]\\
\partial\left(\Omega^{n_{i+1}}_i(x) \cup L_{n_{i+1}} \cap U_{n_{i+1}} \right)& x \in [a_{n_{i+1}},b_{n_{i+1}}]\\
\partial \tilde \cV_{U_{n_{i+1}},U_n}(x) & x \in [b_{n_{i+1}},a_n+1]\\
\partial\left(\Omega^n_i(x-1) \cup U_n \right)& x \in[a_n+1,b_n+1]. 
\end{cases}
\]
Here, the $\tilde \cV$ are the homotopies $\cV$ reparametrized to be defined on the given intervals (the exact parametrization is immaterial). It is easy to check that $\tilde\Phi^n_{i+1}$ is continuous. 

Let $\Phi_{i+1}^n$ denote the reparametrization of $\tilde\Phi_{i+1}^n$ by volume. We have arranged that $\Phi_{i+1}^n$ is a $J_{n}$-replacement. Moreover, for $x \in [a_{i+1}',b_{i+1}']$, we have that $\Phi^n_{i+1}(x) = \Phi^{n_{i+1}}_i(x)$, so 
\[
\bM(\Phi^n_{i+1}(x)) \leq A_j
\]
for $x \in [a_j',b_j']$ and $1\leq j\leq i$. Finally, we can set 
\[
A_{i+1} : = \sup_{x\in[a_{i+1}',b_{i+1}']} \bM(\Phi^{n_{i+1}}_{i}(x)) < W
\]
(which is independent of $n$). This completes the proof. 
\end{proof}

\subsection{Proof Theorem \ref{c:non-excessive_minmax}}
We are now able to complete proof of Theorem \ref{c:non-excessive_minmax}

		Let $\Phi$ be a nested optimal sweepout. Consider the collection $\cA$ of the maximal (with respect to inclusion) excessive intervals for $\Phi$, that is $I\in \cA$ if for every excessive interval $I'$ such that $I'\cap I\neq \emptyset$, we have $I\supset I'$. The existence of maximal intervals follows from Proposition \ref{p:max_int} proven above. 
	
	    Notice that by definition $I \neq J\in \cA$ implies that $I\cap J=\emptyset$, so we can define a new sweepout $\Psi$ in the following way
\[
	    \Psi(x)=\begin{cases}
	    \Phi^I(x) & \quad \text{if }x \in I \in \cA\\
	    \Phi(x) & \quad \text{otherwise}\,.
	    \end{cases}
\]
Note that $\Psi$ is a nested optimal sweepout, so up to reparametrization we can assume it is (ONVP), and moreover by construction $\mm(\Psi)\subset \mm(\Phi)$. Suppose that $x \in \mm_L(\Psi)$ is left excessive. Then, there is a $\Psi$-excessive interval $J$ with $(x-\eps,x]\subset J$. We claim that there is $I \in \cA$ with $J \subset I$. Indeed, if $J \cap I = \emptyset$ for all $I \in \cA$, then $J$ is a $\Phi$-excessive interval, contradicting the definition of $\cA$. On the other hand, if there is $I \in \cA$ with $J \cap I \not =\emptyset$, then $J \cup I$ is excessive by Lemma \ref{l:union-excessive-repeated}. Thus, $J \subset I$ by definition of $\cA$ again. Thus, for $y \in (x-\eps,x]\subset I$, $\Psi(y) = \Psi^I(y)$. By the definition of replacement family, we know that if $x_i \in (x-\eps,x]$ has $x_i\to x$, then
\[
\limsup_{i\to\infty} \bM(\Psi^I(x_i)) < W. 
\]
However, this contradicts the assumption that $x \in\mm_L(\Psi)$. The same proof works to prove that $x \in \mm_R(\Psi)$ is not right excessive. This finishes the proof. \qed

\section{Deformation Theorems and Proof of Theorem \ref{thm:nm+index_bound}}\label{ss:deformation_thm}
In this section we conclude the proof of Theorem \ref{thm:nm+index_bound}. By Theorem \ref{c:non-excessive_minmax}, there exists an (ONVP) sweepout $\Phi$ so that every $x \in \mm_L(\Phi)$ is not left excessive and every $x\in\mm_R(\Phi)$ is not right excessive. By Almgren--Pitts pull-tight and regularity theory \cite{Pitts}, we find that for some $x_0\in\mm(\Phi)$, there is a min-max sequence $x_i\to x_0$ so that $|\Phi(x_i)|$ converges to some $V \in \cR$. Indeed, we can pull-tight $\Phi$ to find a sweepout (in the sense of Almgren--Pitts, not in the (ONVP) sense considered in this paper) $\tilde\Phi$; we have that $\C(\tilde\Phi) \subset \C(\Phi)$ and some $V\in\C(\tilde\Phi)$ is in $\cR$. By replacing
$\Phi(x)$ by $\Phi(1-x)$ if necessary,
we can then assume for the rest of this section that:
\begin{equation}\label{e:no_can}
\begin{gathered}
\text{there is a (ONVP) sweepout $\{\Phi(x)=\partial \Omega(x)\}$ and $x_i \nearrow x_0 \in \mm_L(\Phi)$, so that}\\
\text{$|\Phi(x_i)| \to V\in\cR$ and $\Phi$ is not left excessive at $x_0$}
\end{gathered}
\end{equation}

We then consider two cases: $\bM(\Phi(x_0)) = W$ (no cancellation) and $\bM(\Phi(x_0)) < W$ (cancellation). We analyze the geometric properties of $V$ in both cases separately,
proving deformation theorems reminiscent of those in \cite{MN16}. 
	
\subsection{No cancellation} 
Throughout this subsection we will assume the no cancellation condition
\[
\bM(\Phi(x_0)) = W \,.
\] 
In this case we have that $|\Phi(x_i)| \to |\partial \Omega|$, see for instance \cite[Proposition A.1]{DL}, so we can rephrase our assumption \eqref{e:no_can} as
\begin{equation}\label{e:no_canc2}
\begin{gathered}
\text{there is a (ONVP) sweepout $\{\Phi(x)=\partial \Omega(x)\}$ and $x_i \nearrow x_0 \in \mm_L(\Phi)$, so that} \\
\text{$|\Phi(x_i)| \to |\Sigma|:=|\partial \Omega|\in\cR$ and $\Phi$ is not left excessive at $x_0$.}
\end{gathered}
\end{equation}
In particular, in this case the multiplicity bound of Theorem \ref{thm:nm+index_bound} follows immediately.




\begin{proposition}
\label{p:pairs}
Let $\Sigma$ be as in \eqref{e:no_canc2}. 
Suppose $\Sigma$ is not homotopic minimizing to either side
in some  open set $U$. Then the following holds:
\begin{enumerate}
    \item for every $x \not\in \overline{U}$ there exists $r>0$, such that
$\Sigma$ is minimizing to one side in $B_r(x)$;
    \item for every open set $U'$ disjoint from $U$,
    we have that $\Sigma$ is homotopic minimizing
    to one side in $U'$.
\end{enumerate}

\end{proposition}	
	\begin{proof} 
We prove statement (1).
There is $\delta>0$ and Caccioppoli sets $E^-_1 \in \cI(\Omega\,|\,U)$ and $ E^+_1\in \cO(\Omega\,|\,U)$ with 
	\begin{equation}\label{e:vol_dec}
	\Per(E^\pm_1\,|\, U)\leq \Per(\Omega\,|\, U)-\delta\,,
	\end{equation}
	and nested families $\{\Omega^-_1(x)\,:x\in [0,1]\} \in \cI(\Omega, E_1^-\,|\,U)$ and $\{\Omega^+_1(x)\,:x\in [0,1]\} \in \cO(\Omega, E_1^+\,|\,U)$. Furthermore, by Lemma \ref{l:existence_homotopic}, we can assume that $E_1^+$ are inner and $E_1^-$ are outer homotopic minimizers in $U$.
	
    Let $x \in \Sigma \setminus \overline{U}$ and assume, for contradiction,
	that $\Sigma$ is not area minimizing on both sides 
	in every ball $B_r(x)$, $r< {\rm dist}(x, U)$. Let $E^-_2 \subset \Omega$, 
	with $\Omega \setminus  E^-_2 \subset B_r(x)$, denote a Caccioppoli set
	that is a strict outer minimizer in $\Omega \cap B_r(x)$.
	Similarly, let $ \Omega \subset E^+_2$, with
	$ E^+_2 \setminus \Omega \subset B_r(x)$, denote a Caccioppoli set
	that is a strict inner minimizer in $\Omega \cap B_r(x)$.
    We have $$\Per(\Omega)> \max\{\Per(E^\pm_2) \}.$$
	If we choose $r>0$ sufficiently small,
	then, by Lemma \ref{l:close in flat}, there exist nested families $\{\Omega^-_2(x)\,:x\in [0,1]\}$ and $\{\Omega^+_2(x)\,:x\in [0,1]\} $ that interpolate between $E^-_2$ and 
	$\Omega$ and between $\Omega$ and $E^+_2$ and satisfying
\begin{equation}\label{e:vol_dec2}
\Per(\Omega^\pm_2(x))\leq \Per(\Omega) + \frac{\delta}{2}
\end{equation}

	Let $(x_l,x_r)\neq \emptyset$ be the interval (since $\Phi$ is nested) such that
	$$
	\Phi(x) \cap (\cup_i E_i^+\setminus \cup_i E_i^-)\neq \emptyset \, .
	$$
	Then we define a family $\bar \Psi\colon [x_l-2, x_r+2] \to \mathcal{Z}_{n}(M; \mathbb{Z}_2)$ by setting
	$$
	\bar \Psi(x):=
	\begin{cases}
	\partial\left(\Omega(x+2)\cap E_1^-\cap E_2^-\right) & \text{if } x\in (x_l-2,x_0-2]\\
	\partial \left(\Omega_1^-(x-x_0+2)\cap E_2^- \right) & \text{if } x\in [x_0-2,x_0-1]\\
	\partial \left(\Omega_1^+(x-x_0+1)\cap E_2^-\right) & \text{if } x\in [x_0-1,x_0]\\
	\partial \left(\Omega_2^-(x-x_0)\cup E_1^+\right) & \text{if } x\in [x_0,x_0+1]\\
	\partial \left(\Omega_2^+(x-x_0-1) \cup E_1^+\right) & \text{if } x\in [x_0+1,x_0+2]\\
	\partial\left(\Omega(x-2)\cup E_1^+\cup E_2^+\right) & \text{if } x\in [x_0+2,x_r+2)
	\end{cases}
	$$
	It is easy to see that $\bar \Psi$ is continuous, and moreover notice that, since by Lemma \ref{strict minimizer} $E_1^+$ is a strict inner minimizer in $U$ and $E_{1}^-$ strict outer minimizers in $U$, we have that
	$$
	\limsup_{y\to x} \bM (\bar \Psi(y)) < \limsup_{y\to x} \bM(\bar\Phi(y)) \leq W 
	$$
	for $x\in(x_l-2,x_0-2] \cup [x_0-2,x_r-2)$.
	Since the families $\Omega_1^\pm(x)$ do not increase the volume of $\Sigma$ in $U_i$ and using \eqref{e:vol_dec} and \eqref{e:vol_dec2}, we also have
	$$
	\bM (\bar \Psi(x)) \leq W -\frac{\delta}{2} \qquad \forall \, x\in [x_0-2, x_0+2]\,.
	$$
	We let $\Psi$ be the volume reparametrization of the nested sweepout $\bar \Psi$, then $\Psi$ is a $(x_l,x_r)$-replacement for $\Phi$, thus giving a contradiction with the fact that $x_0\in (x_l,x_r)$ and $x_0\in \mm_L(\Phi)$.
	
	The proof of statement (2) is completely analogous.
\end{proof}


\begin{proposition}
\label{p:def_thm}
Let $\Sigma$ be as in \eqref{e:no_canc2}, then the following holds
\begin{enumerate}
    \item $\Index(\Sigma)\leq 1$;
    \item If $\Index(\Sigma)=1$, then 
 for every point $x \in \Sigma$ there exists $r>0$, such that
$\Sigma$ is minimizing to one side in $B_r(x)$;
    \item If $\hnm(\Sigma)$ is non-empty, then $\Sigma$ is stable,
    $\cH^0(\hnm(\Sigma))=1$ and for every point $x \in \Sigma\setminus \hnm(\Sigma)$ there exists $r>0$, such that $\Sigma$ is minimizing to one side in $B_r(x)$.
\end{enumerate}
In particular, Theorem \ref{thm:nm+index_bound} holds in the case of no cancellations.
\end{proposition}
	
\begin{proof} 
Note that if $U \cap \Sigma$ is smooth and unstable, it is easy to see that $\Sigma$ is not homotopic minimizing to either side in $U$ (just consider the normal flow generated by a compactly supported unstable variation of fixed sign). Statements (2) and (3)
of the Proposition now immediately follow from Proposition \ref{p:pairs}. 
The upper bound on the index (1) follows from (2) of Proposition \ref{p:pairs} and Lemma \ref{l:localizing-index-two-sided} below. 
\end{proof}	
\begin{lemma}[Localizing the index]\label{l:localizing-index-two-sided}
Suppose that $\Sigma \in \cR$ is two-sided and has $\Index(\Sigma) \geq 2$. Then, there is $\Sigma^*_1,\Sigma^*_2 \subset \Sigma$ smooth hypersurfaces with boundary so that the $\Sigma^*_i$ are both unstable (for variations fixing the boundary). 
\end{lemma}
\begin{proof}
A standard capacity argument implies that there is a subset $\Sigma' \subset \Sigma$ where $\Sigma'$ is a smooth minimal surface with smooth boundary and $\Index(\Sigma') \geq 2$ (with Dirichlet boundary conditions). Let $u$ denote the second (Dirichlet) eigenfunction (with eigenvalue $\lambda <0$) for the stability operator for $\Sigma$. Because $u$ must change sign, there are at least two nodal domains $\Sigma_1,\Sigma_2\subset \Sigma$. One can find subsets with smooth boundary $\Sigma^*_i\subset \Sigma_i$ so that $\Sigma^*_i$ are unstable. This follows from the argument in \cite[p.\ 21]{Chavel} (namely, by considering $(u|_{\Sigma_i} - \eps)_+$ in the stability operator for $\eps\to 0$ chosen so that $\{u|_{\Sigma_i} > \eps\}$ has smooth boundary). 
\end{proof}

\begin{lemma}\label{lem:snm-hnm}
$\Snm(\Sigma)\subset \mathfrak{h}_\textrm{nm}(\Sigma)$.
\end{lemma}
\begin{proof}
Suppose that $p \in \Snm(V)$, we claim that $\Sigma$ is not homotopic minimizing to either side in $B_\eps(p)$ for any $\eps>0$ sufficiently small. Indeed, by assumption, the unique tangent cone $\bC = \partial\Omega_\bC$ to $\Sigma$ at $p$ is not minimizing to either side. This implies that there are Caccioppoli sets $E_\bC^- \subset \Omega_\bC \subset E_\bC^+$ so that $E_\bC^\pm \Delta\Omega_\bC \subset B_1 \subset \RR^{n+1}$ and so that
\[
\Per_{\RR^{n+1}}(E_\bC^\pm\,|\,B_1) \leq \Per_{\RR^{n+1}}(\Omega_\bC\,|\,B_1) - \delta.
\]
Choose $C^{1,\omega}$ coordinates on $M$ around $p$ so that $\Omega = \Omega_\bC$ in $B_\eps(p)$ and so that $g_{ij}(p) = \delta_{ij}$, which we can do since $g\in C^2$ and $\Sigma$ is a $C^{1,\omega}$ deformation of $\bC$ near $p$ by assumption. Then, set 
\[
E(x) := \begin{cases}
 (\Omega\setminus B_\eps) \cup (|x| E_\bC^- \cap B_\eps) & x < 0\\
 \Omega & x = 0\\
 (\Omega\setminus B_\eps) \cup (|x| E_\bC^+ \cap B_\eps) & x > 0\\
 \end{cases}
\]
We have that 
\[
\Per_g(E(x)) - \Per_g(\Omega) = - |x|^n \delta (1+o(1))
\]
as $x\to 0$ (since the metric $g_{ij}$ converges to the flat metric $\delta_{ij}$ after rescaling $|x| \to 1$, by the $C^{1,\omega}$ regularity of the chart). This shows that $\Sigma$ is not homotopic minimizing to either side in $B_\eps(p)$, so $p\in\mathfrak{h}_\textrm{nm}(\Sigma)$ as claimed. 
\end{proof}

\subsection{Cancellation} We will assume the cancellation condition
\[
\bM(\Phi(x_0)) < W
\]
throughout this subsection. In particular, we can find $q \in \reg V$ so that for all $\eps>0$ sufficiently small,
\[
\Per(\Omega\,|\,B_\eps(q)) < | V |(B_\eps(q)) 
\]
where $\partial \Omega = \Phi(x_0)$. Like in the previous section we set $\Sigma := \supp V$. 

Furthermore we set $V=\sum_{i}\kappa_i\,|\Sigma_i|$, where each $\Sigma_i$ is a minimal hypersurface with optimal regularity and $\kappa_i\in \NN$ are constant multiplicities, by the constancy theorem \cite[Theorem 41.1]{Sim}. So \eqref{e:no_can} becomes

\begin{equation}
\label{e:canc}
\begin{gathered}
\text{there is a (ONVP) sweepout $\{\Phi(x)=\partial \Omega(x)\}$ and $x_i \nearrow x_0 \in \mm_L(\Phi)$, so that} \\
\text{$|\Phi(x_i)| \to V=\sum_{i} \kappa_i\,|\Sigma_i|\in\cR$, $\Phi$ is not left excessive at $x_0$ and} \\
\text{there is $q\in \Sigma$ such that }\, \Per(\Omega\,|\,B_\eps(q)) \leq | V |(B_\eps(q))-\delta(\eps) \quad \text{for all } \eps>0\,.
\end{gathered}
\end{equation}

We write $\Omega=\Omega(x_0)$ and observe that $\Sigma\subset \overline\Omega$. We would like to claim that $\Sigma$ is homotopically minimizing, but this condition might not make sense if $\Sigma$ is one-sided. However, thanks to the cancellation we can actually prove 
that $\Sigma$ is area-minimizing in its neighborhood in $\Omega$
away from a small ball around $q$.

\begin{definition} \label{def:nbhd minimizing}
We will call a set $\Omega'$ a $(q,\eps, \tau, \Sigma, \Omega)$-competitor
if $$\big(\Omega \setminus B_\tau(\Sigma)\big)  \cup \big(B_\eps(q) \setminus \Sigma \big) 
\subset \Omega' \subsetneqq 
\Omega \setminus \Sigma \, .$$
An $(q,\eps, \tau, \Sigma, \Omega)$-competitor $\Omega'$ will be called a
minimizing competitor if 
its perimeter is strictly less than perimeter of any 
$(q,\eps, \tau, \Sigma, \Omega)$-competitor $\Omega''$ 
with $\Omega' \subset \Omega''$.
(Note that we do not require $\Per (\Omega')$ to 
be less that the perimeter of all competitors, but
only those that contain $\Omega'$).
\end{definition}



\begin{proposition} \label{thm:no_competitor}
    Suppose \eqref{e:canc} holds, then 
    for every $\eps>0$ there is $\tau>0$,
    such that 
    minimizing $(q,\eps, \tau, \Sigma, \Omega)$-competitor
    does not exist.
\end{proposition}

\begin{proof} 
For contradiction suppose 
there exists a minimizing $(q,\eps, \tau, \Sigma, \Omega)$-competitor
$U$. Observe that by the cancellation
assumption for every $\eta>0$ 
we can find $(q,\eps, \tau, \Sigma, \Omega)$-competitors 
$\Omega'$
with $\Per(\Omega') \leq W+ \eta - \delta(\epsilon)$,
where $\delta(\epsilon)$ is from (\ref{e:canc}). 
It follows that 
$$\Per(U) \leq \Per(\Omega')<W \, .$$
If we choose $\tau>0$ sufficiently small,
then by Lemma \ref{l:close in flat}
there exists a nested family $\{E(x)\,:\,x\in [0,1]\}$
with $E(0)= U$, $E(1) = \Omega$ and
$$\Per(E(x))< W \, .$$

Let $(x_l,x_0]$ be the connected interval such that $\Omega(x)\setminus U\neq \emptyset$, where $\{\Phi(x)=\partial \Omega(x)\}$, and define 
family $\Psi \colon (x_l, x_0+1] \to \mathcal Z_n(M, \mathbb Z_2)$ by
$$
\Psi(x):=
\begin{cases}
\partial(\Omega(x)\cap U) & \text{if }x\in (x_l,x_0] \\
\partial  E(x-x_0)
& \text{if }x\in [x_0, x_0+1]
\end{cases}
$$
Clearly $\Psi$ is continuous, since $\Omega=\Omega(x_0)$ and moreover we have that
$$
\limsup_{y\to x} \bM(\Psi(y)) < \limsup_{y\to x} \bM(\Phi(x)) \leq W
$$
for every $x\in (x_l,x_0)$ by
strict minimality condition in Definition \ref{def:nbhd minimizing}. For every $x\in [x_0, x_0+1]$
we also have $\bM(\Psi(x)) = \bM(\partial E(x))<W$. This implies that $x_0$ is left excessive for $\Phi$ which is a contradiction.
\end{proof}

\begin{proposition}\label{p:def-thm-cancel}
Suppose $V=\sum_i\kappa_i \,|\Sigma_i|$ is as in \eqref{e:canc}, then each $\Sigma_i$ has stable regular part and $\hnm(V) = \emptyset$.
Moreover, for every point 
$x \in support(V)$
there exists $r>0$, such that 
the support of $V$ is minimizing to one side in $B_r(x)$
\end{proposition}

\begin{proof} 
First we observe that we can find two points $q_1$ and $q_2$ in $\reg V$,
such that for all $\eps>0$ sufficiently small,
\[
\Per(\Omega\,|\,B_\eps(q_j)) < | V |(B_\eps(q)) \, .
\]
By Proposition \ref{thm:no_competitor} we have non-existence
of minimizing $(q_j,\eps, \tau, \Sigma, \Omega)$-competitors
for $j=1,2$. This implies that 
 $\Sigma_i$ is area minimizing to one side
in a small ball around every point of $V$. In particular,
we have $\cH^{0}(\hnm(V)))=0$.

The stability of the regular part of each $\Sigma_i$ also follows from the
non-existence of minimizing $(q,\eps, \tau, \Sigma, \Omega)$-competitors.
Indeed, if a component $\Sigma_i$ has index $\geq 1$,
then for $\eps>0$ sufficiently small, the  minimal
hypersurface $\Sigma_i \setminus B_{\eps}(q)$ with fixed boundary
will be unstable by a standard capacity argument.
If $\Sigma_i$ is two-sided, then by considering a minimization problem
to one side of $\Sigma_i$ 
in $B_\tau(\Sigma_i) \setminus B_\eps(q)$
we can find open set $U \subset \Omega$,
such that $\Omega \setminus U$ is a minimizing 
$(q,\eps, \tau, \Sigma, \Omega)$-competitor.

Suppose $\Sigma_i$ is one-sided.
Since $\Sigma_i \subset \overline{\Omega}$
we have that $B_\tau(\Sigma_i)\setminus \Sigma_i \subset \Omega$
for all sufficiently small $\tau>0$.
In particular, for small $\tau< \eps$ 
we can minimize in the class of hypersurfaces
 $\{S\subset B_\tau(\Sigma_i) : S \cap B_\eps(q) = \Sigma_i \cap B_\eps(q)\}$
 to obtain a minimizer $\Sigma_i'$ in the same homology class
 and open set $U \subset \Omega$ with $\partial U = \Sigma_i \cup \Sigma_i'$.
 Then $\Omega \setminus U$ is a minimizing 
$(q,\eps, \tau, \Sigma, \Omega)$-competitor.
\end{proof}

\subsection{Multiplicity $2$ bound} In this subsection we show that if $\kappa_i> 2$ for some $i$, then $x_0$ is excessive, by using simple comparisons with disks. Notice that if any multiplicity satisfies $\kappa_i\geq 2$ then we must be in the cancellation case considered above. 

\begin{lemma}[Multiplicity $2$ bound]\label{l:mult_bound}
Let $V=\sum_{i}\kappa_i\,| \Sigma_i|$ be as in \eqref{e:no_can}. Then $\kappa_i\leq 2$ for every $i$.
\end{lemma}
	
\begin{proof}  Suppose by contradiction $\kappa_i\geq 3$ for some $i$. Then let $p\in \reg(\Sigma_i)$, $p\neq q$ (where $q$ is the cancellation point considered above). Consider a ball  $B_r(p)$, $r < \frac{1}{2}dist(p,q)$, sufficiently small so that $\Sigma_i\cap B_r(p)$ is two-sided. Let $\tau(r)>0$ be a small constant
to be chosen later
and set $U = B_r(p) \cap B_\tau(\Sigma_i)$.

Consider sequence $x_j \nearrow x_0$
with $|\partial \Omega(x_j)| \rightarrow V$.
We can assume that the radius $r$ was
 chosen sufficiently small, so that
\begin{equation}\label{e:mult_bound1}
    \bM(\partial \Omega(x_j) \cap U)\geq \left(\kappa_i - \frac{1}{10}\right) \omega_n  r^n\,,
\end{equation}
for all $j$ large enough, where $\omega_n$ denotes the measure of the $n$-dimensional ball of radius one.

Let $\Omega_j' \subset \Omega(x_j)$, $\Omega_j' \setminus U = \Omega(x_j)
\setminus U$, be a strict one-sided outer area minimizer
in $\Omega(x_j)\cap U$.
Observe that if $\Omega'_j$ does not converge to $\Omega(x_0)$,
then $\lim \Omega_j'$ is a $(q,\frac{1}{2}dist(p,q), \tau, \Sigma, \Omega(x_0))$-competitor,
which contradicts Proposition \ref{thm:no_competitor}.

We conclude that $\lim \Omega_j' = \Omega(x_0)$.
On the other hand, 
by comparing $\Omega(x_j) \setminus U$ to $\Omega_j'$ and
assuming that $\tau(r)$
was chosen sufficiently small, we have that one-sided 
area minimizing property of $\Omega_j'$ implies 
\begin{equation*}
    \bM(\partial \Omega_j' \cap U)\leq \Per(U) \leq \left(2 + \frac{1}{10}\right) \omega_n  r^n\,,
\end{equation*}
For $\tau(r)$
 sufficiently small and $j$ large we can apply
Lemma \ref{l:close in flat} to find a nested family $E(x)$
interpolating between $\Omega_j'$ and $\Omega$, such that
\begin{align*}\Per( E(x) ) &\leq \max \{\bM(\partial \Omega_j' \setminus U),
\bM(\partial \Omega(x_0) \setminus U) \} + \left(2 + \frac{2}{10}\right) \omega_n  r^n\\
& \leq W - \left( 1-\frac{3}{10} \right) \omega_n r^n.
\end{align*}
By combining families $\Omega(x) \cap \Omega_j'$
and $E(x)$ we obtain that $x_0$ is left-excessive.
\end{proof}

	\subsection{Proof of Theorem \ref{thm:nm+index_bound}} The result follow immediately by combining Corollary \ref{c:non-excessive_minmax} with Propositions \ref{p:pairs}, \ref{p:def_thm}, \ref{p:def-thm-cancel} and Lemma \ref{l:mult_bound}. \qed
		
	\section{Proof of Theorems \ref{thm:generic_bound_ricci},  \ref{thm:generic_bound}, and \ref{thm:generic_stratum}}
    In this section we prove Theorem \ref{thm:generic_stratum} (Theorems \ref{thm:generic_bound_ricci} and \ref{thm:generic_bound} follow immediately from Theorem \ref{thm:generic_stratum} when combined with the facts that when $n=8$ all singularities are regular and that the set of bumpy metrics is open and dense \cite{White:bumpy,White:bumpy2}). Theorem \ref{thm:generic_stratum} will follow from Theorem \ref{thm:nm+index_bound} and Proposition \ref{prop:min-vs-htpy-min}, together with a simple surgery procedure.

    \subsection{Surgery procedure} We show here how to regularize minimal hypersurfaces with regular singularities under the assumption that the hypersurface minimizes area in a small ball around each singularity. 
    \begin{proposition}[Perturbing away regular singularities of locally area minimizing surfaces]\label{p:surgery}
	For $(M^{n+1},g)$ a compact $C^{2,\alpha}$-Riemannian metric and $\Sigma\in\cR$ a minimal hypersurface, recall that $\cS_0(\Sigma)\subset \sing \Sigma$ is defined to be the set of singular points with a regular tangent cone. There is $\tilde g\in \Met^{2,\alpha}(M)$ arbitrarily close to $g$ and $\tilde \Sigma$ arbitrarily close in the Hausdorff sense to $\Sigma$ so that $\tilde\Sigma$ is minimal with respect to $\Sigma$ and $\cS_0(\tilde\Sigma) \subset \hnm(\tilde\Sigma) = \hnm(\Sigma)$. 
\end{proposition} 
\begin{proof}
For every $p \in \cS_0(\Sigma)\setminus \hnm(\Sigma)$, and $\eps_0=\eps_0(p)$ so that $\Sigma \cap (B_{\eps_0}(p)\setminus p)$ is regular, we will show how to perturb $g$ and $\Sigma$ so that $p$ becomes regular. We will do this by making an arbitrarily small change to $g$, $\Sigma$ supported in $B_{\eps_0}(p)$. Because $\cS_0$ is discrete (but not necessarily closed when $n\geq 9$) it is easy to enumerate the elements of $\cS_0(\Sigma)\setminus \hnm(\Sigma)$ and make a summably small change around each point. As such, it suffices to consider just the perturbation near $p$.

By definition, taking $\eps<\eps_0$ sufficiently small, $\Sigma \cap B_\eps(p)$ is one-sided homotopy area-minimizing. For concreteness write $\Sigma \cap B_\eps(p) = \partial \Omega$ in $B_\eps(p)$ and assume that $\Omega$ is inner homotopy minimizing. By Lemma \ref{lem:snm-hnm}, the tangent cone at $p$ is area-minimizing (to the same side). 

We claim that (after taking $\eps>0$ smaller if necessary) there is a sequence of $\Sigma_i\in\cR(B_\eps(p))$ with stable regular part, with $\Sigma_i \subset \Omega$, $\Sigma_i$ disjoint from $\Sigma$, and $\Sigma_i\to \Sigma$. Indeed, we can apply Proposition \ref{prop:min-vs-htpy-min} to conclude that either (after shrinking $\eps>0$), $\Omega$ is area-minimizing to the inside, or there are $\Sigma_i$ as asserted.

In the case that $\Omega$ is area-minimizing to the inside, we can still construct the $\Sigma_i$ by shrinking $\eps>0$ even further so that $\Omega$ is strictly area-minimizing to the inside and then minimizing area with respect to a boundary
$\Sigma \cap \partial B_\eps(p) + \delta_i$, for a sequence $\delta_i\to 0$;
i.e., the boundary of $\Sigma\cap B_\eps(p)$ pushed slightly into $\Omega$. By the unique minimizing property, the minimizers will converge back to $\Sigma$ in $B_\eps(p)$. 

For $i$ sufficiently large we can write the intersection of $\Sigma_i$
with the annulus $A(p,\eps/5,\eps)$ as a graph of function $u_i$
over $\Sigma$.

Reasoning as in Hardt--Simon \cite[Theorem 5.6]{HS} (cf. \cite[Theorem 3.1]{Liu}), for $i$ sufficiently large, $\Sigma_i$ will be regular in $B_{\eps/2}(p)$. We now set 
\[
\tilde\Sigma_i =(\Sigma_i \cap B_{\eps/5}) \cup (\Sigma\setminus B_\eps(p)) \cup ((\Sigma + \chi u_i)\cap A(p,\eps/5,\eps))
\]
where $\chi$ is a smooth cutoff function with $\chi\equiv1$ on $B_{\eps/5}$ and $\chi\equiv 0$ on $B_{3\eps/5}$. Note that
\[
H_g(\tilde\Sigma_i) \quad \text{is supported in $B_{4\eps/5}(p) \setminus B_{\eps/5}(p)$}
\]
and $\Vert H_g(\tilde \Sigma_i) \Vert_{C^{2,\alpha}} = o(1)$ as 
$i \to \infty$.

	Now, define $\tilde g = e^{2f} g$, in this new metric, since $\tilde \Sigma$ is smooth, we have the transformation
	$$
	H_{\tilde g}(\tilde \Sigma)= e^{-f}\left(H_g(\tilde\Sigma)+\frac{\partial f}{\partial \nu} \right)\,,
	$$
	where $\nu$ is the normal direction to $\Sigma$. 
	Setting $H_{\tilde g}(\tilde \Sigma)=0$, this reduces to the equation
	$$
	H_g(\tilde\Sigma)+\frac{\partial f}{\partial \nu}=0
	$$
	which implies that $f=-H_g(\tilde \Sigma) \zeta(\nu)$, for a function $\zeta(t)$ such that $\zeta'(0)=1$ and $\zeta\equiv0$ for $|t|\geq \eps/100$ is a solution. Since, as observed, $H_g(\tilde\Sigma)$ is supported in $A(p,\eps/5,4\eps/5)$, so is the metric change, and since $\|u\|_{C^{4,\alpha}}\leq o(1)$ and $\chi$ is smooth, we have
	$$
	\|g-\tilde g\|_{2,\alpha}=\|e^f-1\|_{C^{2,\alpha}} \|g\|_{2,\alpha} \leq C\, \|u\|_{C^{4,\alpha}}\,   \|g\|_{2,\alpha} = o(1)
	$$ 
	as $i\to \infty$. This completes the proof. 
\end{proof}
	
	\subsection{Proof of Theorem \ref{thm:generic_stratum}} For $g \in \Met^{2,\alpha}(M)$, apply Theorem \ref{thm:nm+index_bound} to find $V\in\cR$ with
	\[
	\cH^0(\hnm(V)) + \Index(V) \leq 1.
	\]
	We can apply Proposition \ref{p:surgery} to $\Sigma = \supp V$ to find a metric $\tilde g$ that is arbitrarily $C^{2,\alpha}$-close to $g$ and a $\tilde g$ minimal hypersurface $\tilde\Sigma \in \cR$ so that $\cS_0(\Sigma) \subset \hnm(\Sigma)$. (Note that if $\Index(V) = 1$, then $ \hnm(\Sigma) =\emptyset$, so $\cS_0(\Sigma) = \emptyset$.) This completes the first part of the proof. 
	
We now consider $g \in \Met^{2,\alpha}_{\Ric>0}(M)$.\footnote{The idea is that positive Ricci curvature rules out stable hypersurfaces but this requires the hypersurface to be two-sided. As such, we must consider two cases, depending on whether $\Sigma$ is one or two-sided.} If $\Sigma$ is two-sided, then $\Index(\Sigma)\geq1$, so we can argue as above.  On the other hand if $\Sigma$ is one-sided, then $[\Sigma]\neq 0 \in H_{n}(M,\ZZ_{2})$. We can then find $\hat\Sigma \in [\Sigma]$ by minimizing area in the homology class. The surface $\hat\Sigma$ may have singularities, but they are all locally area minimizing. Thus, we can apply Proposition \ref{p:surgery} to $\Sigma$ yielding $\tilde\Sigma$ and $\tilde g$ with $\cS_0(\tilde\Sigma) = \emptyset$. 
\qed

\bibliography{bib}
\bibliographystyle{alpha}


\end{document}